\documentclass[a4paper,11pt]{amsart}
\usepackage{amsmath,amsthm,amssymb}
\usepackage[mathscr]{eucal}
\usepackage{cite}
\usepackage{upgreek}
\usepackage[bookmarks=false]{hyperref}
\usepackage{enumerate}
\usepackage{mathrsfs}
\usepackage{blindtext}
\usepackage{scrextend}
\usepackage{enumitem}
\addtokomafont{labelinglabel}{\sffamily}
\usepackage{color}
\usepackage[at]{easylist}
\usepackage{bbm}

\numberwithin{equation}{section}

\theoremstyle{plain}
\newtheorem{main}{Theorem}
\newtheorem{mcor}[main]{Corollary}
\newtheorem{mprop}[main]{Proposition}

\newtheorem{theorem}{Theorem}[section]

\newtheorem{lemma}[theorem]{Lemma}

\newtheorem{corollary}[theorem]{Corollary}
\theoremstyle{definition}
\newtheorem{definition}[theorem]{Definition}
\newtheorem*{definition*}{Definition}
\newtheorem{example}[theorem]{Example}
\newtheorem{examples}[theorem]{Examples}

\newtheorem{remark}[theorem]{Remark}

\newcommand{\N}{\mathbb{N}}
\newcommand{\R}{\mathbb{R}}
\newcommand{\C}{\mathbb{C}}

\newcommand{\bM}{\mathbb{M}}

\newcommand{\cF}{\mathcal{F}}

\newcommand{\cH}{\mathcal{H}}

\newcommand{\cM}{\mathcal{M}}
\newcommand{\cN}{\mathcal{N}}
\newcommand{\cO}{\mathcal{O}}

\newcommand{\cS}{\mathcal{S}}

\newcommand{\cU}{\mathcal{U}}

\newcommand{\bF}{\mathbb{F}}

\newcommand{\otb}{\bar{\otimes}}
\newcommand{\ot}{\otimes}

\newcommand{\emb}{\prec}

\newcommand{\Ad}{\operatorname{Ad}}

\newcommand{\id}{\operatorname{id}}

\newcommand{\actson}{\curvearrowright}
\newcommand{\eps}{\varepsilon}
\newcommand{\norm}[1]{\left\|#1\right\|}

\newcommand{\cross}{\rtimes}
\renewcommand{\qed}{\hfill$\blacksquare$}

\begin{document}

\title[II$_1$ factors with exotic central sequence algebras]
{II$_1$ factors with exotic central sequence algebras}

\author[A. Ioana]{Adrian Ioana}
\address{Department of Mathematics, University of California San Diego, 9500 Gilman Drive, La Jolla, CA 92093, USA}
\email{aioana@ucsd.edu}

\author[P. Spaas]{Pieter Spaas}
\address{Department of Mathematics, University of California San Diego, 9500 Gilman Drive, La Jolla, CA 92093, USA}
\email{pspaas@ucsd.edu}

\thanks{The authors were supported in part by NSF Career Grant DMS \#1253402.}

\begin{abstract} 

We provide a class of separable II$_1$ factors $M$ whose central sequence algebra is not  the ``tail" algebra associated to any decreasing sequence of von Neumann subalgebras of $M$.
This settles a question of McDuff \cite{Mc69d}.

\end{abstract}

\maketitle

\section{Introduction and statement of main results}

A uniformly bounded sequence $(x_k)$ in a II$_1$ factor $M$ is called {\it central} if $\lim_k\|x_ky-yx_k\|_2=0$, for every $y\in M$. 
Central sequences have played a fundamental role in the study of II$_1$ factors since the very beginning of the subject with Murray and von Neumann's property Gamma \cite{MvN43}.
A separable II$_1$ factor $M$ has {\it property Gamma} if it admits a central sequence $(x_k)$ which is not trivial, in the sense that $\inf_k\|x_k-\tau(x_k)1\|_2>0$. 
Murray and von Neumann proved that the unique hyperfinite II$_1$ factor has property Gamma, while the free group factor $L(\mathbb F_2)$ does not, thus giving the first example of two non-isomorphic separable II$_1$ factors \cite{MvN43}. Over two decades later, in the late 60s,  the analysis of central sequences of \cite{MvN43} was refined to provide additional examples of non-isomorphic separable II$_1$ factors in \cite{Ch69,DL69,Sa68,ZM69},  culminating with McDuff's construction of a continuum of such factors \cite{Mc69a,Mc69b}.  

Shortly after, McDuff \cite{Mc69c} defined the {\it central sequence algebra} of a II$_1$ factor $M$  as the relative commutant, $M'\cap M^{\omega}$, of $M$ into its ultrapower $M^{\omega}$ (\!\!\!\cite{Wr54,Sa62}), where $\omega$ is a free ultrafilter on $\mathbb N$.
  This  has since allowed for a more structural approach to central sequences and led to significant progress in the study of II$_1$ factors.
Indeed, the central sequence algebra was a crucial tool in Connes' famous classification of amenable II$_1$ factors \cite{Co76}. 
Furthermore, the relative commutant $M'\cap\cM^\omega$, for some von Neumann algebra $\cM\supset M$, was used by Popa to formalise his influential spectral gap rigidity principle in \cite{Po06a,Po06b}.  Most recently, central sequence algebras and their subalgebras were used to provide a continuum of II$_1$ factors with non-isomorphic ultrapowers in \cite{BCI15} (adding to the four such factors noticed in \cite{FGL06,FHS11,GH16}).

However, despite the progress the use of central sequence algebras has allowed, their structure remains fairly poorly understood.  For instance, it is open whether any  II$_1$ factor $M$ whose central sequence algebra is abelian admits an abelian subalgebra $A$ such that $M'\cap M^\omega \subset A^\omega$ (see \cite{Ma17}). In this article, we investigate the existence of a certain ``canonical form" for central sequence algebras.
 To make this precise, we recall the following notions introduced by McDuff in \cite{Mc69d} in order to distil the key ideas 
  of \cite{Mc69b}:

\begin{definition}[\!\!{\cite[Definition~2]{Mc69d}}]\label{mdef}
	Let $M$ be a separable II$_1$ factor. A von Neumann subalgebra $A$ of $M$ is called {\it residual} if $\lim_k\|x_k-E_A(x_k)\|_2=0$, for every central sequence $(x_k)$ in $M$.
	A sequence $(A_n)_{n\in\mathbb N}$ of von Neumann subalgebras of $M$ is called a {\it residual sequence} if
	\begin{enumerate}
	\item $A_{n+1}\subset A_n$, for every $n$,
	\item $A_n$ is residual in $M$, for every $n$, and
	\item if $x_k\in A_k$ and $\|x_k\|\leq 1$, for every $k$, then the sequence $(x_k)$ is central in $M$.
	\end{enumerate}
\end{definition}

 \begin{remark} 
 A decreasing sequence $(A_n)_{n\in\mathbb N}$ of von Neumann subalgebras of $M$ is residual if and only if $M'\cap M^{\omega}=\cap_{n\in\mathbb N}A_n^{\omega}$. Thus, a separable II$_1$ factor $M$ admits a residual sequence if and only if its central sequence algebra is equal to the ``tail" algebra, $\cap_{n\in\mathbb N}A_n^{\omega}$, associated to a decreasing sequence of von Neumann subalgebras $(A_n)_{n\in\mathbb N}$. 
\end{remark}
In \cite{Mc69d}, McDuff noted that it was unknown whether every II$_1$ factor admits a residual sequence. 
She gave examples of II$_1$ factors which do not admit any {\it strongly} residual sequence $(A_n)_{n\in\mathbb N}$ (i.e. ones satisfying, in addition to (1)-(3), the existence of a subalgebra $A^n\subset A_n$ such that $A_n=A_{n+1}\bar{\otimes}A^n$), but left open the case of residual sequences. 
The main goal of this article is to provide the first examples of II$_1$ factors with no residual sequence. Before stating our results in this direction, let us note that several large, well-studied classes of II$_1$ factors admit a residual sequence.

\begin{examples}\label{ex}
The following II$_1$ factors admit a residual sequence:

\begin{enumerate}
\item Any II$_1$ factor without property Gamma. 
\item The hyperfinite II$_1$ factor $R$. If we write $R=\otb_{k\in\N} \bM_2(\C)$, and let $R_n=\otb_{k\geq n} \bM_2(\C)$, then $(R_n)_{n\in\N}$ is a residual sequence in $R$.
\item Any II$_1$ factor $M$ which is strongly McDuff, i.e. can be written as $M=N\bar{\otimes}R$, where $N$ is a II$_1$ factor without property Gamma. If  $A_n=1\otimes R_n$, then Connes'  characterization of property Gamma \cite[Theorem 2.1]{Co76} implies that $(A_n)_{n\in\mathbb N}$ is a residual sequence in $M$.
\item Any infinite tensor product $M=\bar{\otimes}_{k\in\mathbb N}M_k$ of II$_1$ factors without property Gamma. If $A_n=\bar{\otimes}_{k\geq n}M_k$, then \cite[Theorem 2.1]{Co76} implies that $(A_n)_{n\in\N}$ is a residual sequence in $M$. Note that $M$ is McDuff, i.e. $M\cong M\bar{\otimes}R$, but not strongly McDuff \cite[Theorem 4.1]{Po09a}.
\item The II$_1$ factors $L(T_0(\Gamma))$ and $L(T_1(\Gamma))$, where $\Gamma$ is any countable group and the countable groups $T_0(\Gamma), T_1(\Gamma)$ are defined as in \cite{DL69,Mc69b} (see also \cite[Section 1.1]{BCI15}). Then $T_0(\Gamma)$ and $T_1(\Gamma)$ both contain $\widetilde\Gamma:=\oplus_{i\in\N}\Gamma_i$, where each $\Gamma_i$ is a copy of $\Gamma$. If $A_n=L(\oplus_{i\geq n}\Gamma_i)$, then \cite[Corollary 2.11]{BCI15} shows that $(A_n)_{n\in\N}$ is a residual sequence in both $L(T_0(\Gamma))$ and $L(T_1(\Gamma))$.  In particular, the uncountably many II$_1$ factors which were shown to have non-isomorphic ultrapowers in \cite{BCI15} all admit residual sequences.
\item Any tensor product $M=\bar{\otimes}_{k=1}^NM_k$, where $N\in\N\cup\{\infty\}$, and for every $k$, $M_k$ is a II$_1$ factor admitting a residual sequence, $(A_{k,n})_{n\in\mathbb N}$. If $B_n=(\bar{\otimes}_{k=1}^{\min\{n,N\}}A_{k,n})\bar{\otimes}(\bar{\otimes}_{k=\min\{n,N\}+1}^NM_k)$, then \cite[Proposition 5.2]{Ma17} implies that $(B_n)_{n\in\mathbb N}$ is a residual sequence in $M$.

\end{enumerate}
\end{examples}

\begin{remark} 
 In \cite{Po09a,Po09b}, Popa studied the class of II$_1$ factors $M$ which arise as an inductive limit of subfactors $(M_n)$ with spectral gap and noticed that $M'\cap M^\omega = \cap_n (M_n'\cap M)^\omega$ (see \cite[Lemma 2.3]{Po09a}). Thus, every such II$_1$ factor $M$ admits a residual sequence, $(M_n'\cap M)_{n\in\N}$. Conversely, although it is unclear whether any II$_1$ factor admitting a residual sequence must be an inductive limit of subfactors with spectral gap, we note that this holds for the factors in Examples \ref{ex} (1)-(5). \end{remark}

We are now ready to state our first main result which gives examples of II$_1$ factors with no residual sequences, and thereby settles McDuff's question \cite{Mc69d}. 
\begin{main}\label{A}
Let $\Gamma$ be a countable non-amenable group. For every $k\in\mathbb N$, let $\pi_k:\Gamma\rightarrow\mathcal O(\mathcal H_k)$ be an orthogonal representation such that \begin{enumerate} \item $\pi_k^{\otimes l}$ is weakly contained in the left regular representation of $\Gamma$, for some $l=l(k)\in\mathbb N$, and
\item there is an orthonormal sequence $(\xi_k^m)_{m\in\mathbb N}\subset\mathcal H_k$ such that $\sup_{m\in\mathbb N}\|\pi_k(g)(\xi_k^m)-\xi_k^m\|\rightarrow 0$, as $k\rightarrow\infty$, for every $g\in\Gamma$. \end{enumerate}

Let $\Gamma\curvearrowright (B_k,\tau_k)$ be the Gaussian action associated to $\pi_k$, and $\Gamma\curvearrowright (B,\tau):=\bar{\otimes}_{k\in\mathbb N}(B_k,\tau_k)$ be the diagonal product action. Define  $M=B\rtimes\Gamma$.

Then the II$_1$ factor $M$ does not admit a residual sequence of von Neumann subalgebras.
\end{main}

For the definition of Gaussian actions, we refer the reader to Section \ref{gauss}.
Next, we provide a class of examples to which Theorem \ref{A} applies, and discuss a connection with a problem posed in \cite{JS85}.

\begin{example}\label{free} Let $\Gamma=\mathbb F_n$ be the free group on $n\geq 2$ generators.
Denote by $|g|$ the word length of an element $g\in\Gamma$ with respect to a free set of generators. Let $t>0$. By \cite{Ha79},  the function $\varphi_t:\Gamma\rightarrow\mathbb R$ given by $\varphi_t(g)=e^{-t|g|}$ is positive definite. Let $\rho_t:\Gamma\rightarrow\mathcal O(\mathcal H_t)$ be the GNS orthogonal representation associated to $\varphi_t$ and $\xi_t\in\mathcal H_t$ such that $\langle\rho_t(g)(\xi_t),\xi_t\rangle=\varphi_t(g)$, for all $g\in\Gamma$. 
Let  $\tilde{\rho}_t=\rho_t\otimes\text{Id}_{\ell^2(\N)}:\Gamma\rightarrow\mathcal O(\mathcal H_t\otimes\ell^2(\N))$ be the direct sum of infinitely many copies of $\rho_t$. 

Let $(t_k)$ be any sequence of positive numbers converging to $0$ and put $\pi_k:=\tilde{\rho}_{t_k}:\Gamma\rightarrow\mathcal O(\mathcal H_{t_k}\otimes\ell^2(\N))$.  Then the representations $(\pi_k)_{k\in\N}$ satisfy the hypothesis of Theorem \ref{A}.
Firstly, given $t>0$, note that $\varphi_t^l\in\ell^2(\Gamma)$, and hence $\rho_t^{\otimes l}$ is contained in a multiple of the left regular representation of $\Gamma$,  whenever $l>\log(2n-1)/(2t)$.
This implies that $\pi_k^{\otimes l}$ is contained in a multiple of the left regular representation of $\Gamma$, for some integer $l=l(k)\geq 1$.
Secondly, note that the vectors $\xi_k^m:=\xi_{t_k}\otimes\delta_m\in\mathcal H_{t_k}\otimes\ell^2(\N)$ satisfy $\sup_{m\in\mathbb N}\|\pi_k(g)(\xi_k^m)-\xi_k^m\|=\sqrt{2(1-\varphi_{t_k}(g))}\rightarrow 0$, as $k\rightarrow\infty$, for any $g\in\Gamma$. 
\end{example}

\begin{remark}
Theorem \ref{A} also sheds new light on a problem of Jones and Schmidt.  In \cite[Theorem 2.1]{JS85}, they proved that any ergodic but not strongly ergodic countable measure preserving
equivalence relation $\mathcal R$ on a probability space $(X,\mu)$ admits a hyperfinite quotient. More specifically, there exists an ergodic hyperfinite measure preserving equivalence relation $\mathcal R_{\text{hyp}}$ on a probability space $(Y,\nu)$ together with a factor map $\pi:(X,\mu)\rightarrow (Y,\nu)$ such that $(\pi\times\pi)(\mathcal R)=\mathcal R_{\text{hyp}}$, almost everywhere.
In \cite[Problem 4.3]{JS85}, Jones and Schmidt asked whether there is always such a quotient with the additional property that $\mathcal R_0:=\{(x_1,x_2)\in\mathcal R\mid\pi(x_1)=\pi(x_2)\}$ is strongly ergodic on almost all of its ergodic components.
If such a quotient exists, then following \cite[Definition 1.3]{IS18} we say that $\mathcal R$ has the {\it Jones-Schmidt property}.
If $\mathcal R$ has the Jones-Schmidt property and we let $M=L(\mathcal R)$, $A=L^{\infty}(X)$, then there exists a decreasing sequence of von Neumann subalgebras $(B_n)_{n\in\N}$ of $A$ such that $M'\cap A^{\omega}=\cap_nB_n^{\omega}$ and $B_{n+1}\subset B_n$ has finite index for every $n\in\N$ (see \cite[Proposition 5.3 and the proof of Lemma 6.1]{IS18}).

In \cite[Theorems E and F]{IS18}, the authors settled in the negative \cite[Problem 4.3]{JS85} by providing examples of equivalence relations $\mathcal R$ without the Jones-Schmidt property. This was achieved by showing that for certain $\mathcal R$, in the above notation, $M'\cap A^{\omega}$ is not equal to $\cap_n B_n^{\omega}$, for any decreasing sequence of von Neumann subalgebras $(B_n)_{n\in\N}$ of $A$ with $B_{n+1}\subset B_n$ of finite index for every $n\in\N$. 

Theorem \ref{A} allows us to strengthen the negative solution to \cite[Problem 4.3]{JS85} given in \cite{IS18}.  More precisely, in the context of Theorem \ref{A}, assume that $\Gamma$ is not inner amenable and let $\mathcal R$ be the equivalence relation associated to the action $\Gamma\curvearrowright B$.
Since $M=L(\mathcal R)=B\rtimes\Gamma$ has no residual sequence by Theorem \ref{A}, while $M'\cap A^{\omega}=M'\cap M^{\omega}$ by \cite{Ch82}, we deduce that $M'\cap A^{\omega}$ cannot be written as $\cap_nB_n^{\omega}$, for {\it any} decreasing sequence $(B_n)_{n\in\N}$ of von Neumann subalgebras of $A$.
\end{remark}

Our second main result shows that the conclusion of Theorem \ref{A} also holds if we replace Gaussian by free Bogoljubov actions (see Section \ref{gauss}). Moreover, we establish the following stronger statement:
\begin{main}\label{B}
Let $\Gamma$ be a countable non-inner amenable group. For every $k\in\mathbb N$, let $\pi_k:\Gamma\rightarrow\mathcal O(\mathcal H_k)$ be an orthogonal representation such that \begin{enumerate} \item $\pi_k^{\otimes l}$ is weakly contained in the left regular representation of $\Gamma$, for some $l=l(k)\in\mathbb N$, and
\item there are orthogonal unit vectors $\xi_k^1,\xi_k^2\in\mathcal H_k$ such that $\max_{m\in\{1,2\}}\|\pi_k(g)(\xi_k^m)-\xi_k^m\|\rightarrow 0$, as $k\rightarrow\infty$, for every $g\in\Gamma$. \end{enumerate}

Let $\Gamma\curvearrowright (B_k,\tau_k)$ be the free Bogoljubov action associated to $\pi_k$, and $\Gamma\curvearrowright (B,\tau):=\bar{\otimes}_{k\in\mathbb N}(B_k,\tau_k)$ be the diagonal product action. Define $M=B\rtimes\Gamma$.

Then the II$_1$ factor $M$ does not admit a residual sequence of von Neumann subalgebras.

 Moreover, there exists a separable von Neumann subalgebra $P\subset M'\cap M^{\omega}$ such that there is no sequence $(A_n)_{n\in\mathbb N}$ of von Neumann subalgebras of $M$ satisfying $P\subset \prod_{\omega} A_n\subset M'\cap M^{\omega}$. 
\end{main}

Since $\Gamma=\mathbb F_n$ is not inner amenable for any $n\geq 2$, and the representations $(\pi_k)_{k\in\N}$ from Example \ref{free} satisfy the hypothesis of Theorem \ref{B}, its conclusion holds for those examples. Moreover, in the notation from  Example \ref{free}, $\pi_k=\rho_{t_k}\oplus\rho_{t_k}$ also satisfy the hypothesis of Theorem \ref{B}.

In order to put Theorem \ref{B} into a better perspective and to contrast it with Theorem \ref{A}, we note the following result:

\begin{mprop}\label{C}
	Let $(M_n,\tau_n)$, $n\in\mathbb N$, be a sequence of tracial von Neumann algebras. Let $P$, $Q$ be commuting separable von Neumann subalgebras of $\prod_{\omega}M_n$. Assume that $P$ is  amenable. 
	
	Then there exist commuting von Neumann subalgebras $P_n,Q_n$ of $M_n$, for every $n\in\mathbb N$, such that $P\subset\prod_{\omega}P_n$ and $Q\subset\prod_{\omega}Q_n$.
\end{mprop}

Proposition \ref{C} implies that for any tracial von Neumann algebra $(M,\tau)$ and any separable amenable von Neumann subalgebra $P\subset M'\cap M^{\omega}$, there is a sequence $(P_n)_{n\in\N}$ of von Neumann subalgebras of $M$ such that $P\subset\prod_{\omega}P_n$ and $M\subset\prod_{\omega}(P_n'\cap M)$, and therefore $P\subset\prod_nP_n\subset M'\cap M^{\omega}$. Consequently, the moreover part of Theorem \ref{B} cannot hold if $P$ is amenable. In particular, if $M=B\rtimes\Gamma$ is as in Theorem \ref{A} and $\Gamma$ is not inner amenable, then $M$ will not satisfy the moreover assertion of Theorem \ref{B}. Indeed, in this case $M'\cap M^{\omega}$ is abelian, being a subalgebra of $B^{\omega}$ by \cite{Ch82}.

In recent years there has been growing interest in the study of the notion of stability for groups (see the survey \cite{Th18}).
As a byproduct of the methods developed in this article, we obtain two applications to the notion of tracial stability for countable groups, formalised recently in \cite{HS17} (see also \cite{HS16}):

\begin{definition}[\!\!\!{\cite[Definition~3]{HS17}}]
	A countable group  $\Gamma$ is  \textit{$W^*$-tracially stable}  if for any sequence $(M_n,\tau_n)$, $n\in\N$, of tracial von Neumann algebras  and any homomorphism $\varphi:\Gamma\rightarrow\mathcal U(\prod_{\omega}M_n)$, there exist homomorphisms $\varphi_n:\Gamma\rightarrow \mathcal U(M_n)$, for every $n\in\N$, such that $\varphi = (\varphi_n)_n$.
\end{definition}
The class of W$^*$-tracially stable groups contains all abelian and free groups, as well as other classes of both amenable and non-amenable groups, see \cite{HS17}.
As an immediate consequence of Proposition \ref{C}, we deduce that the class of W$^*$-tracially stable groups is closed under taking the direct product with an amenable group. For the case of the direct product with an abelian group, this result is part of \cite[Theorem 1]{HS17}.

\begin{mcor}\label{D}
	Let $\Gamma$ and $\Sigma$ be $W^*$-tracially stable groups. Assume that $\Sigma$ is amenable. Then $\Gamma\times\Sigma$ is $W^*$-tracially stable.
\end{mcor}

In contrast to Corollary \ref{D}, we show that any direct product of non-abelian free groups is not W$^*$-tracially stable, thereby answering a question of Atkinson in the negative (see \cite[Question~4.16]{At18}). 

\begin{main}\label{E}
	$\bF_l\times\bF_m$ is not $W^*$-tracially stable, for any $2\leq l,m\leq+\infty$. 
	
	Moreover, there exist a II$_1$ factor $M$ and a trace preserving $*$-homomorphism $\varphi:L(\mathbb F_2\times\mathbb F_2)\rightarrow M^{\omega}$ such that there is no sequence of homomorphisms $\varphi_n:\mathbb F_2\times\mathbb F_2\rightarrow\mathcal U(M)$ satisfying $\varphi_{|\mathbb F_2\times\mathbb F_2}=(\varphi_n)_n$.
\end{main}

\subsection*{Structure of the paper} Besides the introduction there are four other sections in this paper. In Section~2 we recall some preliminaries and prove a few useful lemmas needed in the remainder of the paper. In Section~3, inspired by Boutonnet's work \cite{Bo12,Bo14}, we prove a structural result concerning II$_1$ factors associated to Gaussian and free Bogoljubov actions. In Section~4 this is used to prove Theorems~\ref{A} and \ref{B}. Finally in Section~5 we prove Proposition~\ref{C} and use the established machinery from the previous sections to deduce Theorem~\ref{E}.

\subsection*{Acknowledgements} We are very grateful to Scott Atkinson for bringing \cite[Question~4.16]{At18} to our attention and for stimulating discussions on tracial stability.

\section{Preliminaries}

\subsection{Tracial von Neumann algebras} 
We begin this section by recalling several notions and constructions involving tracial von Neumann algebras.

A {\it tracial von Neumann algebra} 
$(M,\tau)$ is  a von Neumann algebra $M$ equipped with a faithful normal tracial state $\tau:M\rightarrow\mathbb C$. We denote by $L^2(M)$ the completion of $M$ with respect to the $2$-norm $\|x\|_2=\sqrt{\tau(x^*x)}$ and consider the standard representation $M\subset\mathbb B(L^2(M))$. 
We also denote by $\mathcal U(M)$ the group of unitary elements of $M$, by $(M)_1=\{x\in M\mid\|x\|\leq 1\}$ the unit ball of $M$, and by $\mathcal Z(M)=M\cap M'$ the center of $M$. It follows from von Neumann's bicommutant theorem, that a self-adjoint set $S\subset M$ generates $M$ as a von Neumann algebra if and only if $S''=M$.

Let $P\subset M$ be a unital von Neumann subalgebra. {\it Jones' basic construction} of the inclusion $P\subset M$ is defined as the von Neumann subalgebra of $\mathbb B(L^2(M))$ generated by $M$ and the orthogonal projection $e_P:L^2(M)\rightarrow L^2(P)$, and is denoted by $\langle M,e_P\rangle$. The basic construction $\langle M,e_P\rangle$ carries a canonical semi-finite trace $\hat \tau$ defined by $\hat \tau(xe_Py)=\tau(xy)$, for all $x,y\in M$. We further denote by $E_P:M\rightarrow P$ the  conditional expectation onto $P$, by $P'\cap M=\{x\in M\mid\text{$xy=yx$, for all $y\in P$}\}$ the {\it relative commutant of $P$ in $M$}, and by $\mathcal N_{M}(P)=\{u\in\mathcal U(M)\mid uPu^*=P\}$
the {\it normalizer of $P$ in $M$}. We say that $P$ is {\it regular} in $M$ if $\mathcal N_M(P)$ generates $M$ as a von Neumann algebra.

Any trace preserving action $\Gamma\curvearrowright^{\sigma}(M,\tau)$ extends to a unitary representation $\sigma:\Gamma\rightarrow\mathcal U(L^2(M))$ called the {\it Koopman representation of $\sigma$}.

Let $\omega$ be a free ultrafilter on $\mathbb N$. Consider the C$^*$-algebra $\ell^{\infty}(\mathbb N,M)=\{(x_n)\in M^{\mathbb N}\mid\sup\|x_n\|<\infty\}$ together with its closed ideal $\mathcal I=\{(x_n)\in\ell^{\infty}(\mathbb N,M)\mid\lim\limits_{n\rightarrow\omega}\|x_n\|_2=0\}.$  Then $M^{\omega}:=\ell^{\infty}(\mathbb N,M)/\mathcal I$ is a tracial von Neumann algebra, called the {\it ultrapower } of $M$, whose canonical trace is given by $\tau_{\omega}(x)=\lim\limits_{n\rightarrow\omega}\tau(x_n)$, for all $x=(x_n)\in M^{\omega}$. If $(M_n)_n$ is a sequence of von Neumann subalgebras of $M$, then their {\it ultraproduct}, denoted by $\prod_\omega M_n$, can be realized as the von Neumann subalgebra of $M^{\omega}$ consisting of $x=(x_n)$ such that $\lim\limits_{n\rightarrow\omega}\|x_n-E_{M_n}(x_n)\|_2=0$.

\begin{lemma}
\label{increase}
Let $(M,\tau)$ be a tracial von Neumann algebra and $(A_n)_n$ be a sequence of von Neumann subalgebras of $M$ such that  $\prod_{\omega}A_n\subset M'\cap M^{\omega}$.
Then $\lim\limits_{n\rightarrow\omega}\|x-E_{A_n'\cap M}(x)\|_2=0$, for every $x\in M$. 

\end{lemma}

{\it Proof.} 
Let $x\in M$. If $n\in\mathbb N$, we can find $u_n\in\mathcal U(A_n)$ such that $\|x-u_nxu_n^*\|_2\geq\|x-E_{A_n'\cap M}(x)\|_2$ (see, e.g., the proof of \cite[Theorem 2.5]{IS18}). Since $(u_n)\in\prod_{\omega}A_n$ and $\prod_{\omega}A_n\subset M'\cap M^{\omega}$, we get that $\lim_{n\rightarrow\omega}\|x-u_nxu_n^*\|_2=0$ and hence $\lim_{n\rightarrow\omega}\|x-E_{A_n'\cap M}(x)\|_2=0$. \hfill$\blacksquare$

\subsection{Hilbert bimodules}\label{ssec:Hilb} Let $(M_1,\tau_1)$ and $(M_2,\tau_2)$ be two tracial von Neumann algebras. An {\it $M_1$-$M_2$-bimodule} is a Hilbert space $\mathcal H$ endowed with two normal, commuting $*$-homomorphisms $\pi_1:M_1\rightarrow\mathbb B(\mathcal H)$ and $\pi_2:M_2^{\text{op}}\rightarrow\mathbb B(\mathcal H)$. 
We define a $*$-homomorphism $\pi_{\mathcal H}:M_1\otimes M_2^{\text{op}}\rightarrow\mathbb B(\mathcal H)$ by $\pi_{\mathcal H}(x\otimes y^{\text{op}})=\pi_1(x)\pi_2(y^{\text{op}})$ and write $x\xi y=\pi_1(x)\pi_2(y^{\text{op}})\xi$, for all $x\in M_1$, $y\in M_2$ and $\xi\in\mathcal H$. 
We also write $_{M_1}{\mathcal H}_{M_2}$ to indicate that $\mathcal H$ is an $M_1$-$M_2$-bimodule. Examples of bimodules include the {\it trivial} $M_1$-bimodule $_{M_1}L^2(M_1)_{M_1}$ and the {\it coarse} $M_1$-$M_2$-bimodule $_{M_1}L^2(M_1)\otimes L^2(M_2)_{M_2}$.

Next, we recall a few notions and constructions involving bimodules (see \cite[Appendix B]{Co94} and \cite{Po86}).
If $\mathcal H$ and $\mathcal K$ are $M_1$-$M_2$-bimodules, we say that $\mathcal H$ is {\it weakly contained} in $\mathcal K$ and write $\mathcal H\subset_{\text{weak}}\mathcal K$ if $\|\pi_{\mathcal H}(T)\|\leq\|\pi_{\mathcal K}(T)\|$, for all $T\in M_1\otimes M_2^{\text{op}}$. 
If $\mathcal H$ is an $M_1$-$M_2$-bimodule and $\mathcal K$ is an $M_2$-$M_3$-bimodule, then 
the {\it Connes fusion tensor product of $\mathcal H$ and $\mathcal K$} is an $M_1$-$M_3$-bimodule denoted by $\mathcal H\otimes_{M_2}\mathcal K$.
If $\Phi:M_1\rightarrow M_2$ is a unital normal completely positive map, then there is a unique $M_1$-$M_2$-bimodule, denoted by $\mathcal H_{\Phi}$, with a unit vector $\xi_{\Phi}\in\mathcal H_{\Phi}$ such that $M_1\xi_{\Phi} M_2$ is dense in $\mathcal H_{\Phi}$ and $\langle x\xi_{\Phi} y,\xi_{\Phi}\rangle=\tau_2(\Phi(x)y)$, for all $x\in M_1$ and $y\in M_2$. 
The next result analyzes the Connes fusion tensor product of bimodules associated to completely positive maps:

\begin{lemma}\label{compose}
Let $\Phi:M_1\rightarrow M_2$ and $\Psi:M_2\rightarrow M_3$ be unital normal completely positive maps, where $(M_1,\tau_1), (M_2,\tau_2), (M_3,\tau_3)$ are tracial von Neumann algebras. Then the following hold:
\begin{enumerate}
\item  The $M_1$-$M_3$-bimodule $\mathcal H_{\Psi\circ\text{Ad}(u)\circ\Phi}$ is isomorphic to a sub-bimodule of $\mathcal H_{\Phi}\otimes_{M_2}\mathcal H_{\Psi}$, for every $u\in\mathcal U(M_2)$.
\item If $\mathcal U$ is a set of unitaries in $M_2$  whose span is $\|.\|_2$-dense in $M_2$, then the $M_1$-$M_3$-bimodule $\mathcal H_{\Phi}\otimes_{M_2}\mathcal H_{\Psi}$ is isomorphic to a sub-bimodule of $\oplus_{u\in\mathcal U}\mathcal H_{\Psi\circ\text{Ad}(u)\circ\Phi}$.
\end{enumerate}
\end{lemma}

{\it Proof.} For $u\in\mathcal U(M_2)$, we denote $\eta_u:=\xi_{\Phi}u^*\otimes_{M_2}\xi_{\Psi}\in \mathcal H_{\Phi}\otimes_{M_2}\mathcal H_{\Psi}$. Following \cite[Section 1.3.1]{Po86}, for every $x\in M_1,y\in M_3$, we have that $$\langle x\eta_uy,\eta_u\rangle=\langle x\xi_{\Phi}u^*\otimes_{M_2} \xi_{\Psi}y,\xi_{\Phi}u^*\otimes_{M_2}\xi_{\Psi}\rangle=\langle x\xi_{\Phi}u^*p,\xi_{\Phi}u^*\rangle=\tau_2(\Phi(x)u^*pu),$$ where $p\in M_2$ is such that $\tau_2(zp)=\langle z\xi_{\Psi}y,\xi_{\Psi}\rangle=\tau_3(\Psi(z)y)$, for all $z\in M_2$. Thus, for all $x\in M_1,y\in M_3$ we have that $\langle x\eta_uy,\eta_u\rangle=\tau_2(u\Phi(x)u^*p)=\tau_3(\Psi(u\Phi(x)u^*)y)$. This shows that the $M_1$-$M_3$-bimodule $\overline{M_1\eta_uM_3}$ is isomorphic to $\mathcal H_{\Psi\circ\text{Ad}(u)\circ\Phi}$ and proves the first assertion of the lemma.

Finally,  note that if the span of $\mathcal U\subset \mathcal U(M_2)$ is $\|.\|_2$-dense in $M_2$, then the span of $\{\overline{M_1\eta_uM_3}\mid u\in\mathcal U\}$ is dense in $\mathcal H_{\Phi}\otimes_{M_2}\mathcal H_{\Psi}$. This implies the second assertion.
\hfill$\blacksquare$

\subsection {Intertwining-by-bimodules} We next recall from  \cite[Theorem 2.1 and Corollary 2.3]{Po03} the powerful {\it intertwining-by-bimodules} technique of Popa.

\begin {theorem}[\!\!\cite{Po03}]\label{corner} Let $(M,\tau)$ be a tracial von Neumann algebra and $P\subset pMp, Q\subset qMq$ be unital von Neumann subalgebras, for some projections $p,q\in M$. Then the following conditions are equivalent:

\begin{itemize}

\item There exist projections $p_0\in P, q_0\in Q$, a $*$-homomorphism $\theta:p_0Pp_0\rightarrow q_0Qq_0$  and a non-zero partial isometry $v\in q_0Mp_0$ such that $\theta(x)v=vx$, for all $x\in p_0Pp_0$.

\item There is no net $u_n\in\mathcal U(P)$ satisfying $\|E_Q(x^*u_ny)\|_2\rightarrow 0$, for all $x,y\in pMq$.

\item There exists a non-zero projection $f\in P'\cap \langle M,e_Q\rangle$ with $\hat \tau(f)<\infty$.
\end{itemize}

If one of these conditions holds true,  then we write $P\prec_{M}Q$, and say that {\it a corner of $P$ embeds into $Q$ inside $M$.}
If $Pp'\prec_{M}Q$ for any non-zero projection $p'\in P'\cap pMp$, then we write $P\prec^{s}_{M}Q$.
\end{theorem}

\subsection {Amenability} A tracial von Neumann algebra $(M,\tau)$ is called {\it amenable} if there exists a positive linear functional $\varphi:\mathbb B(L^2(M))\rightarrow\mathbb C$ such that $\varphi_{|M}=\tau$ and $\varphi$ is $M$-{\it central}, in the following sense:  $\varphi(xT)=\varphi(Tx)$, for all $x\in M$ and $T\in\mathbb B(L^2(M))$. Equivalently, $(M,\tau)$ is amenable if $_ML^2(M)_M$ is weakly contained in $_ML^2(M)\otimes L^2(M)_M$.
By Connes' celebrated classification of amenable factors \cite{Co76}, $M$ is amenable if and only if it is approximately finite dimensional.

Next, we recall the notion of relative amenability introduced by Ozawa and Popa. 
Let $p\in M$ be a projection, and
 $P\subset pMp, Q\subset M$ be von Neumann subalgebras. 
Following \cite[Section 2.2]{OP07} we say that $P$ is  {\it amenable relative to $Q$ inside $M$} if there exists a positive linear functional $\varphi:p\langle M,e_Q\rangle p\rightarrow\mathbb C$ such that $\varphi_{|pMp}=\tau$ and $\varphi$ is $P$-central.  

As shown in \cite[Lemma 2.7]{DHI16}, relative amenability is closed under inductive limits.
Here we establish the following generalization of this result, which we will need later on.
Given a set $I$, we denote by $\lim_{n}$ a state on $\ell^{\infty}(I)$ which extends the usual limit.

\begin{lemma}\label{union}
Let $(M,\tau)$ be a tracial von Neumann algebra and $P,Q\subset M$ be von Neumann subalgebras. Assume that $P_n\subset M$, $n\in I$, is a net of von Neumann subalgebras such that $\|E_{P_n}(x)-x\|_2\rightarrow 0$, for all $x\in P$, and $p_n\in P_n'\cap M$ are projections such that $P_np_n$ is amenable relative to $Q$ inside $M$, for every $n\in I$.
Then there exists a projection $p\in P'\cap M$ such that $Pp$ is amenable relative to $Q$ inside $M$ and $\tau(p)\geq\lim_n\tau(p_n)$.
\end{lemma}
\begin{proof}
We may clearly assume that $c:=\lim_n\tau(p_n)>0$ and $\tau(p_n)>0$, for every $n$.
For  every $n$, let $\varphi_n:p_n\langle M,e_Q\rangle p_n\rightarrow\mathbb C$ be a $P_np_n$-central positive linear functional such that ${\varphi_n}_{|p_nMp_n}=\tau$.
 The Cauchy-Schwarz inequality implies that $|\varphi_n(p_nTxp_n)|\leq \sqrt{\varphi_n(p_nTT^*p_n)\varphi_n(p_nx^*xp_n)}\leq \|T\|\|x\|_2$, and similarly that $|\varphi_n(p_nxTp_n)|\leq \|T\|\|x\|_2$, for all $x\in M$, $T\in \langle M,e_Q\rangle$.

We define a state $\varphi:\langle M,e_Q\rangle\rightarrow\mathbb C$ by letting 
$$\text{$\varphi(T)=\lim_n\frac{\varphi_n(p_nTp_n)}{\tau(p_n)}$, for every $T\in\langle M,e_Q\rangle$.}$$

We claim that $\varphi$ is $P$-central. To this end, let $x\in P$, $T\in\langle M,e_Q\rangle$ and $n\in I$. Since $\varphi_n$ is $P_np_n$-central,  $\varphi_n(p_nTE_{P_n}(x)p_n)=\varphi_n(p_nE_{P_n}(x)Tp_n)$ and thus \begin{align*}|\varphi_n(p_nTxp_n)-\varphi_n(p_nxTp_n)|&\leq |\varphi_n(p_nT(x-E_{P_n}(x))p_n)|+|\varphi_n(p_n(x-E_{P_n}(x))Tp_n)|\\&\leq2\|T\|\|x-E_{P_n}(x)\|_2.  \end{align*}
Since $\|x-E_{P_n}(x)\|_2\rightarrow 0$ and $\lim_n\tau(p_n)>0$, we get that $\varphi(Tx)=\varphi(xT)$, and the claim is proven.

Finally, note that $\varphi_{|M}\leq \frac{1}{c}\tau$. Thus, we can find $y\in P'\cap M$ such that $0\leq y\leq\frac{1}{c}$ and $\varphi(x)=\tau(xy)$, for all $x\in M$. Let $p\in P'\cap M$ be the support projection of $y$. Then $y\leq \frac{1}{c}p$, hence $\tau(p)\geq c\tau(y)=c$. Since the restriction of $\varphi$ to $p(P'\cap M)p$ is faithful, \cite[Theorem 2.1]{OP07} implies that $Pp$ is amenable relative to $Q$ inside $M$, which finishes the proof.
\end{proof}

\begin{corollary}\label{amenable}
Let $(M,\tau)$ and $(N,\tau')$ be tracial von Neumann algebras. Assume that there exists a net of von Neumann subalgebras $P_n\subset M$, $n\in I$, and trace preserving $*$-homomorphisms $\pi_n:N\rightarrow M$ such that $\|\pi_n(x)-E_{P_n}(\pi_n(x))\|_2\rightarrow 0$, for every $x\in N$. For $n\in I$, let $p_n\in P_n'\cap M$ be a projection such that $P_np_n$ is amenable. Then there is a projection $z\in\mathcal Z(N)$ such that $Nz$ is amenable and $\tau(z)\geq\lim_n\tau(p_n)$. In particular, if $P_n$ is amenable for every $n$, then $N$ is amenable.
\end{corollary}

{\it Proof.}
For every $n$, let $M_n=M$ and view $P_n$ and $N$  as subalgebras of $M_n$, via the identity map and $\pi_n$, respectively. If  we put $\tilde M=*_{N,n\in I}M_n$, then we have $\|E_{P_n}(x)-x\|_2\rightarrow 0$, for every $x\in N$.   Since $P_np_n$ is amenable for every $n$, Lemma \ref{union} implies the existence of a projection $p\in N'\cap \tilde M$ such that $Np$ is amenable and $\tau(p)\geq\lim_n\tau(p_n)$. Thus, if $z$ is the support projection of $E_{\mathcal Z(N)}(p)$, then $Nz$ is amenable. Since $z\geq p$, we  have that $\tau(z)\geq\tau(p)$, which finishes the proof.
\hfill$\blacksquare$

The next Lemma, which appears to be of independent interest, provides general conditions which guarantee that if $P$ is amenable relative to a decreasing net of subalgebras $Q_n$, then $P$ is amenable relative to their intersection, $\cap_nQ_n$. More generally, we have:

\begin{lemma}\label{generalamen}
Let $(M,\tau)$ be a tracial von Neumann algebra and $Q\subset M$ a von Neumann subalgebra. Assume that there exist nets of von Neumann subalgebras $Q_n,M_n\subset M$ such that \begin{enumerate}
\item $Q\subset M_n\cap Q_n$ and $_{Q_n}L^2(M)_{M_n}\subset_{\text{weak}}
{_{Q_n}L^2(Q_n)}\otimes_{Q} L^2(M_n)_{M_n}$, for every $n$,
\item $\|x-E_{M_n}(x)\|_2\rightarrow 0$, for every $x\in M$.
\end{enumerate}
If $P\subset M$ is a von Neumann subalgebra which is amenable relative to $Q_n$ inside $M$, for every $n$, then $P$ is amenable relative to $Q$ inside $M$.
\end{lemma}

Lemma \ref{generalamen} applies in particular if there exists $u_n\in\mathcal U(M)$ such that $u_nPu_n^*\subset Q_n$, or, more generally, if $P\prec_{M}^{s}Q_n$, for every $n$. Indeed, by \cite[Lemma 2.6(3)]{DHI16}, the latter condition implies that $P$ is amenable relative to $Q_n$ inside $M$.

{\it Proof.} 
Assume that $P$ is amenable relative to $Q_n$, for every $n$. Then \cite[Theorem~2.1]{OP07} gives that
$_PL^2(M)_M\subset_{\text{weak}}{_PL^2(M)}\otimes_{Q_n}L^2(M)_M$, and thus $_PL^2(M)_{M_n}\subset_{\text{weak}}{_{P}L^2(M)}\otimes_{Q_n}L^2(M)_{M_n}$, for every $n$.
Since $_{Q_n}L^2(M)_{M_n}\subset_{\text{weak}}{_{Q_n}L^2(Q_n)}\otimes_Q L^2(M_n)_{M_n}$, we further get that $_PL^2(M)_{M_n}\subset_{\text{weak}}{_P}L^2(M)\otimes_Q L^2(M_n)_{M_n}$, and thus \begin{align*}_{P}L^2(M)\otimes_{M_n}L^2(M)_M&\subset_{\text{weak}}{_PL^2(M)}\otimes_Q L^2(M_n)\otimes_{M_n}L^2(M)_M\\&= {_{P}L^2(M)}\otimes_Q L^2(M)_M,\;\;\text{for every $n$}.\end{align*}
On the other hand, since  $\|x-E_{M_n}(x)\|_2\rightarrow 0$, for every $x\in M$, we have $$_PL^2(M)_M\subset_{\text{weak}}{\bigoplus_n}\, {_PL^2(M)}\otimes_{M_n}L^2(M)_M.$$
By combining the last two displayed inclusions, we get that $_PL^2(M)_M\subset_\text{weak}{_PL^2(M)}\otimes_Q L^2(M)_M$, and therefore $P$ is amenable relative to $Q$ inside $M$.
\hfill$\blacksquare$

\begin{remark} Several weaker versions of particular cases of Lemma \ref{generalamen} have been observed before. Indeed,
conditions (1) and (2) from Lemma \ref{generalamen} are satisfied in the two following cases:
\begin{enumerate}[label=(\alph*)]
\item $M=*_{Q,k\in\mathbb N}M_k$ is an amalgamated free product of tracial von Neumann algebras $(M_k)_{k\in\mathbb N}$ over a common subalgebra $Q$, $Q_n=*_{Q,k\geq n}M_k$ and $M_n=*_{Q,k<n}M_k$.
\item $M=(\bar{\otimes}_{k\in\mathbb N}M_k)\bar{\otimes}Q$ is an infinite tensor product of tracial von Neumann algebras $(M_k)_{k\in\mathbb N}$ and $Q$, $Q_n=(\bar{\otimes}_{k\geq n}M_k)\bar{\otimes}Q$ and $M_n=(\bar{\otimes}_{k<n}M_k)\bar{\otimes}Q$.
\end{enumerate}

Lemma \ref{generalamen} was first noticed by the first author in case (a) under the assumption that $P$ can be unitarily conjugated into $Q_n$, and extended in \cite[Proposition 4.2]{HU15} to cover the more general assumption that $P\prec_M^{s}Q_n$. When $Q=\mathbb C1$, the latter result was also noticed by R. Boutonnet and S. Vaes (personal communication), whose proof inspired our Lemma \ref{generalamen}.
In case (b), weaker versions of Lemma \ref{generalamen} were obtained in \cite[Lemma 4.4]{Is16} and \cite[Proposition 2.7]{CU18}.
\end{remark}

\subsection{Malleable deformations} In \cite{Po01,Po03}, Popa introduced the notion of an {\it s-malleable deformation} of a von Neumann algebra. In combination with his powerful {\it deformation/rigidity} techniques, this notion has led to remarkable progress in the theory of von Neumann algebras (see, e.g., \cite{Po07,Va10a,Io18}). S-malleable deformations will also play an important role in this paper.

\begin{definition}\label{def:smal}
	Let $(M,\tau)$ be a tracial von Neumann algebra. We say that a triple $(\tilde M,(\alpha_t)_{t\in\mathbb R},\beta)$ is an {\it s-malleable deformation} of $M$ if the following conditions hold:  \begin{enumerate}\item  $(\tilde M,\tilde\tau)$ is a tracial von Neumann algebra such that $\tilde M\supset M$ and $\tilde\tau_{|M}=\tau$,
		\item $(\alpha_t)_{t\in\mathbb R}\subset\text{Aut}(\tilde M,\tilde\tau)$ is a $1$-parameter group with $\lim_{t\rightarrow 0}\|\alpha_t(x)-x\|_2=0$, for all $x\in \tilde M$.
		\item $\beta\in\text{Aut}(\tilde M,\tilde\tau)$ satisfies $\beta^2=\text{Id}_{\tilde M}$, $\beta\alpha_t\beta^{-1}=\alpha_{-t}$ for all $t\in\mathbb R$, and $\beta(x)=x$, for all $x\in M$.
	\end{enumerate}
\end{definition}

As established in \cite{Po06a}, s-malleable deformations have the following ``transversality" property:

\begin{lemma}[\!\!{\cite[Lemma~2.1]{Po06a}}]\label{trans}
	For any $x\in M$ and $t\in\R$ we have
	\[
	\norm{x-\alpha_{2t}(x)}_2\leq 2\norm{\alpha_t(x) - E_M(\alpha_t(x))}_2.
	\]	
\end{lemma}

\subsection{Gaussian and free Bogoljubov actions}\label{gauss} We next discuss two kinds of actions that will play a crucial role in this paper, Gaussian and free Bogoljubov actions. 
Below we describe one possible construction of these actions, following \cite{PS09} and \cite{VDN92}. For further properties of Gaussian and free Bogoljubov actions, we refer the reader to \cite{Bo14} and \cite{Ho12b}, respectively.

For the remainder of the preliminaries, we fix an orthogonal representation $\pi:\Gamma\rightarrow \cO(H_\R)$ of a countable group $\Gamma$ on a real Hilbert space $H_\R$. Let $H=H_\R\otimes_\R \C$ be the complexified Hilbert space, $H^{\otimes n}$ its $n^{th}$ tensor power, and $H^{\odot n}$ its symmetric $n^{th}$ tensor power. The latter is the closed subspace of $H^{\otimes n}$ spanned by vectors of the form
\[
\xi_1\odot\dots\odot\xi_n := \frac{1}{n!} \sum_{\sigma\in S_n} \xi_{\sigma(1)}\otimes\dots\otimes\xi_{\sigma(n)},
\]
with the inner product normalized such that $\norm{\xi}^2_{H^{\odot n}} = n! \norm{\xi}^2_{H^{\otimes n}}$. We then consider the {\it symmetric Fock space}
\[
\cS(H) := \C\Omega \oplus \bigoplus_{n\geq 1} H^{\odot n},
\]
where the unit vector $\Omega$ is the so-called {\it vacuum vector}. Any vector $\xi\in H$ gives rise to an unbounded operator $\ell_\xi$ on $\cS(H)$, the so-called {\it left creation operator}, defined by
\[
\ell_\xi(\Omega)=\xi, \text{ and }\, \ell_\xi(\xi_1\odot\dots\odot\xi_n) = \xi\odot\xi_1\odot\dots\odot\xi_n.
\]
Denoting $s(\xi)=\ell_\xi+\ell_\xi^*$, one checks that the operators $\{s(\xi)\}_{\xi\in H}$ commute. Moreover, one can show (\!\!\cite{PS09}) that with respect to the vacuum state $\langle\cdot\Omega,\Omega\rangle$, they can be regarded as independent random variables with Gaussian distribution $\cN(0,\norm{\xi}^2)$.

Consider the abelian von Neumann algebra $A_{\pi}\subset B(\cS(H))$ generated by all operators of the form
\[
\omega(\xi_1,\dots,\xi_n) := \exp(i\pi s(\xi_1)\dots s(\xi_n)),
\]
together with the trace $\tau=\langle \cdot\Omega,\Omega\rangle$. Any orthogonal operator $T\in\cO(H_\R)$ can also be viewed as a unitary operator on its complexification $H$, and gives rise to a unitary operator on $\cS(H)$, which we will still denote by $T$, defined by
\[
T(\Omega)=\Omega, \text{ and }\, T(\xi_1\odot\dots\odot\xi_n) = (T\xi_1)\odot\dots\odot(T\xi_n).
\]
One then checks that $T\omega(\xi_1,\dots,\xi_n)T^* = \omega(T\xi_1,\dots,T\xi_n)$, hence $T$ normalizes $A_{\pi}$. Since $T(\Omega)=\Omega$, $\Ad(T)$ is a trace preserving automorphism of $A_{\pi}$.

\begin{definition}
	The {\it Gaussian action} associated to $\pi$ is the action $\sigma=\sigma_\pi: \Gamma\actson (A_{\pi},\tau)$ defined by $\sigma_g=\Ad(\pi(g))$, for every $g\in\Gamma$.
\end{definition}

One can easily check that the unitaries $\omega(\xi)$ satisfy the properties $\omega(0)=1$, $\omega(\xi+\eta)=\omega(\xi)\omega(\eta)$, $\tau(\omega(\xi)) = \exp(-\norm{\xi}^2)$, and $\sigma_g(\omega(\xi)) = \omega(\pi(g)\xi)$ for all $\xi,\eta\in H, g\in\Gamma$. This in fact gives an equivalent description of the Gaussian action (see \cite{Va10b}).

The free Bogoljubov action arises in a similar way using the {\it full Fock space}
\[
\cF(H) := \C\Omega \oplus \bigoplus_{n\geq 1} H^{\otimes n}.
\]
We consider the {\it left creation operator} $L_\xi$ associated to $\xi\in H$ defined by
\[
L_\xi(\Omega)=\xi, \text{ and }\, L_\xi(\xi_1\otimes\dots\otimes\xi_n) = \xi\otimes\xi_1\otimes\dots\otimes\xi_n.
\]
Putting $W(\xi)=L_\xi + L_\xi^*$, one can show (\!\!\cite{VDN92}) that the distribution of the self-adjoint operator $W(\xi)$ with respect to the vacuum state $\langle\cdot\Omega,\Omega\rangle$ is the semicircular law supported on $[-2\norm{\xi},2\norm{\xi}]$, and that for any orthogonal set of vectors from $H_\R$, the associated family of operators is freely independent with respect to $\langle\cdot\Omega,\Omega\rangle$. 

Denote by $\Gamma(H_\R)''$ the von Neumann algebra generated by $\{W(\xi)\mid \xi\in H_\R\}$. Then $\Gamma(H_\R)''$ is isomorphic to the free group factor $L(\bF_{\dim(H_\R)})$. Moreover, $\tau=\langle\cdot\Omega,\Omega\rangle$ is a normal faithful trace on $\Gamma(H_\R)''$. As for the symmetric Fock space, any operator $T\in\cO(H_\R)$ induces an operator $T\in\cU(\cF(H))$, satisfying $\Ad(T)(W(\xi))=W(T\xi)$.

\begin{definition}
	The {\it free Bogoljubov action} associated to $\pi$ is the action $\rho=\rho_\pi: \Gamma\actson (\Gamma(H_\R)'',\tau)$ defined by $\rho_g=\Ad(\pi(g))$, for every $g\in\Gamma$.
\end{definition}

Since $\overline{\Gamma(H_{\mathbb R})''\Omega}=\cF(H)$, the Koopman representation associated to $\rho$ of $\Gamma$ on $L^2(\Gamma(H_{\mathbb R})'')$ is isomorphic to the representation of $\Gamma$ on $\cF(H)$. This implies the following fact which will be needed later on:

\begin{lemma}\label{weak} Denote by $\rho_0$ the restriction of the Koopman representation of $\rho$ to $L^2(\Gamma(H_{\mathbb R})'')\ominus\mathbb C1$. If $\pi^{\otimes k}$ is weakly contained in the left regular representation of $\Gamma$, for some $k\in\mathbb N$, then 
$\rho_0^{\otimes k}$ is weakly contained in the left regular representation of $\Gamma$.
\end{lemma}

\subsection{Deformations associated to Gaussian and free Bogoljubov actions}\label{ssec:def} We will now recall the construction of s-malleable deformations of the crossed product von Neumann algebras associated to the above actions. On $H_\R\oplus H_\R$ consider the orthogonal operators
\[
A_t = \begin{pmatrix}
\cos(\frac{\pi}{2}t) & -\sin(\frac{\pi}{2}t)\\
\sin(\frac{\pi}{2}t) & \cos(\frac{\pi}{2}t)
\end{pmatrix}, \,t\in\R, \quad\text{and}\quad 
B = \begin{pmatrix}
1&0\\
0&-1
\end{pmatrix}.
\]
We note that canonically, $A_{\pi\oplus\pi}\cong A_\pi\bar{\otimes} A_\pi$ and $\Gamma(H_\R\oplus H_\R)''\cong \Gamma(H_\R)'' \ast \Gamma(H_\R)''$.  Under these identifications, we have that $\sigma_{\pi\oplus\pi}\cong \sigma_\pi\otimes\sigma_\pi$ and $\rho_{\pi\oplus\pi}\cong \rho_\pi\ast\rho_\pi$, respectively. Associated to the operators $A_t$ and $B$ we get automorphisms
\[
\alpha_t := \Ad(A_t),\; t\in\R, \quad\text{and}\quad \beta := \Ad(B)
\]
of $A_{\pi}\bar{\otimes} A_{\pi}$ and $\Gamma(H_\R)''\ast \Gamma(H_\R)''$, respectively. Since $A_t$ and $B$ commute with $\pi\oplus\pi$, it follows that $\alpha_t$ and $\beta$ commute with $\sigma_{\pi}\otimes\sigma_{\pi}$ and $\rho_{\pi}\ast\rho_{\pi}$, respectively.
\begin{itemize}
	\item For the Gaussian action, let $M=A_\pi\cross\Gamma$, $\tilde M = (A_\pi\otb A_\pi)\cross\Gamma$, and view $M$ as a subalgebra of $\tilde M$ via $M\cong (A_\pi\otb 1)\cross\Gamma$. By the discussion above, the automorphisms $\alpha_t$ and $\beta$ of $A_\pi\otb A_\pi$ extend to automorphisms of $\tilde M$ by letting $\alpha_t(u_g)=\beta(u_g)=u_g$, for all $g\in\Gamma$.
	\item For the free Bogoljubov action, let $M=\Gamma(H_\R)''\cross\Gamma$, $\tilde M = (\Gamma(H_\R)''\ast \Gamma(H_\R)'')\cross\Gamma$, and view $M$ as a subalgebra of $\tilde M$ via $M\cong (\Gamma(H_\R)''\ast 1)\cross\Gamma$. By the discussion above, the automorphisms $\alpha_t$ and $\beta$ of $\Gamma(H_\R)''\ast \Gamma(H_\R)''$ extend to automorphisms of $\tilde M$ by letting $\alpha_t(u_g)=\beta(u_g)=u_g$, for all $g\in\Gamma$.
\end{itemize}
In both cases, it is easy to check that $(\tilde M, (\alpha_t)_{t\in\mathbb R},\beta)$ is an $s$-malleable deformation of $M$.

\section{Spectral gap rigidity}

This section is devoted to the following rigidity result and its Corollary \ref{embed}.

\begin{theorem}\label{thm:spgap}
Let $(M,\tau)$  be a tracial von Neumann algebra and $N,P\subset M$ be von Neumann subalgebras. Assume that there exists an $s$-malleable deformation $(\tilde M, (\alpha_t)_{t\in\mathbb R},\beta)$ such that
\begin{enumerate}
\item The $M$-bimodule $\mathcal H:=L^2(\tilde M)\ominus L^2(M)$ has the property that $\mathcal H^{\otimes_M k}$
is weakly contained in the bimodule $L^2(M)\otimes_N L^2(M)$, for some $k\in\mathbb N$.
\item The $M$-bimodule $L^2(\tilde M)$ with the bimodular structure given by $x\cdot \xi\cdot y=x\xi\alpha_1(y)$, for every $x,y\in M,\xi\in L^2(\tilde M)$, is contained in a multiple of the bimodule $L^2(M)\otimes_PL^2(M)$.
\end{enumerate}
Let $Q\subset M$ be a von Neumann subalgebra such that $Qp$ is not amenable relative to $N$ inside $M$, for any non-zero projection $p\in Q'\cap M$. Then $Q'\cap M\prec_{M}^sP$.
\end{theorem}

The proof of Theorem \ref{thm:spgap} relies on Popa's deformation/rigidity theory and notably uses his spectral gap rigidity principle introduced in \cite{Po06a,Po06b}. 
Theorem \ref{thm:spgap} and Corollary \ref{embed} were inspired by Boutonnet's work (see \cite{Bo12} and \cite[Chapter II]{Bo14}), whose exposition we follow closely. Finally,  we note that condition (1) in Theorem \ref{thm:spgap} was first considered by Sinclair in \cite{Si10}.

\begin{corollary}\label{embed}
Let $\Gamma$ be a countable group and $\pi:\Gamma\rightarrow\mathcal O(\mathcal H_{\mathbb R})$ be an orthogonal representation. Assume that $\pi^{\otimes k}$ is weakly contained in the left regular representation of $\Gamma$, for some $k\in\mathbb N$. Let $\Gamma\curvearrowright (C,\tau)$ be either the Gaussian action or the free Bogoljubov action associated to $\pi$. 
Let $\Gamma\curvearrowright (D,\rho)$ be a trace preserving action on a tracial von Neumann algebra $D$, consider the diagonal product action $\Gamma\curvearrowright (C\bar{\otimes}D,\tau\otimes\rho)$, and denote $M=(C\bar{\otimes}D)\rtimes\Gamma$.

Let $Q\subset M$ be a von Neumann subalgebra such that $Qp$ is not amenable relative to $D$ inside $M$, for any non-zero projection $p\in Q'\cap M$. Then $Q'\cap M\prec_{M}^{s}D\rtimes\Gamma$.
\end{corollary}

The remainder of this section is devoted to the proofs of Theorem~\ref{thm:spgap} and Corollary~\ref{embed}.

\begin{lemma}[\!\!\cite{Bo12}]\label{lem:QMomega}
	Let $(\tilde M,\tau)$  be a tracial von Neumann algebra and $N\subset M\subset \tilde M$ be von Neumann subalgebras. Assume that
the $M$-bimodule $\mathcal H:=L^2(\tilde M)\ominus L^2(M)$ has the property that $\mathcal H^{\otimes_M k}$ is weakly contained in the bimodule $L^2(M)\otimes_N L^2(M)$, for some $k\in\mathbb N$.
	
	Let $Q\subset M$ be a von Neumann subalgebra such that $Qp$ is not amenable relative to $N$ inside $M$, for any non-zero projection $p\in Q'\cap M$. Then $Q'\cap \tilde M^\omega \subset M^\omega$. In particular, $Q'\cap \tilde M\subset M$.
\end{lemma}
\begin{proof}
	The proof of \cite[Lemma~2.3]{Bo12}, which applies verbatim for $N=\mathbb C1$, works in general.	
\end{proof}

The following lemma is a standard application of Popa's spectral gap rigidity principle.

\begin{lemma}\label{lem:rigid}
	Let $(M,\tau)$  be a tracial von Neumann algebra and $N\subset M$ be a von Neumann subalgebra. Assume that there exists an s-malleable deformation $(\tilde M, (\alpha_t)_{t\in\mathbb R},\beta)$ such that
	the $M$-bimodule $\mathcal H:=L^2(\tilde M)\ominus L^2(M)$ has the property that $\mathcal H^{\otimes_M k}$ is weakly contained in the bimodule $L^2(M)\otimes_N L^2(M)$, for some $k\in\mathbb N$.

	Let $Q\subset M$ be a von Neumann subalgebra such that $Qp$ is not amenable relative to $N$ inside $M$, for any non-zero projection $p\in Q'\cap M$. Then $\alpha_t$ converges uniformly on $(Q'\cap M)_1$.
\end{lemma}
\begin{proof}
	Fix $\eps>0$. Since $Q'\cap \tilde M^\omega \subset M^\omega$ by Lemma \ref{lem:QMomega}, there exist $x_1,\dots,x_n\in Q$ and $\delta>0$ such that for all $y\in (\tilde M)_1$:
	\[
	\forall i\in\{1,\dots,n\}: \norm{[y,x_i]}_2\leq\delta \quad\Longrightarrow\quad \norm{y-E_M(y)}_2\leq \eps.
	\]
	Taking $t>0$ such that $\norm{\alpha_s(x_i)-x_i}_2\leq \frac{\delta}{2}$ for all $1\leq i\leq n$ and all $s\in [0,t]$, we get for any $x\in (Q'\cap M)_1$ 
	\begin{align*}
	\norm{\alpha_s(x)x_i-x_i\alpha_s(x)}_2 &=\norm{x\alpha_{-s}(x_i)-\alpha_{-s}(x_i)x}_2\\
	&\leq 2\norm{x}\norm{\alpha_{-s}(x_i)-x_i}_2 + \norm{xx_i-x_ix}_2\\
	&\leq 2\norm{\alpha_s(x_i)-x_i}_2\\
	&\leq \delta.
	\end{align*}
	Hence for all $s\in [0,t]$ and $x\in (Q'\cap M)_1$, we have $\norm{\alpha_s(x)-E_M(\alpha_s(x))}_2\leq\eps$ and thus by Lemma~\ref{trans},
$\norm{\alpha_{2s}(x)-x}_2\leq 2\eps$.	It follows that $\alpha_t$ converges uniformly on $(Q'\cap M)_1$.
\end{proof}


\begin{lemma}\label{lem:alpha1impl}
	Assume the setting of Lemma~\ref{lem:rigid} and let $p\in (Q'\cap M)'\cap M$ be a non-zero projection. Then there is a non-zero element $a_1\in p\tilde M\alpha_1(p)$ such that $xa_1=a_1\alpha_1(x)$ for all $x\in (Q'\cap M)p$.
\end{lemma}
\begin{proof} We follow closely the proof of \cite[Theorem 4.1]{Po03}.
	Put $D=Q'\cap M$ and fix a projection $p\in D'\cap M$. %
	
	{\bf Claim 1.}
	For any $t>0$ small enough, there exists a non-zero element $a_t\in p\tilde M\alpha_t(p)$ such that $a_t = ua_t\alpha_t(u^*)$ for all $u\in\cU(Dp)$.
	
	{\it Proof of Claim 1.} By Lemma~\ref{lem:rigid}, $\alpha_t\rightarrow \id$ uniformly on $(Dp)_1$, as $t\rightarrow 0$. Thus, for any $t>0$ small enough we have that $\norm{u-\alpha_t(u)}_2^2 \leq \tau(p)$ and hence
	\begin{equation}\label{eq:Retau}
	\text{$\Re\tau(u\alpha_t(u^*))\geq \frac{\tau(p)}{2}$, for all $u\in\cU(Dp)$}.
	\end{equation}
	Consider the unique element $a_t$ of minimal $\|.\|_2$-norm in the  $\|.\|_2$-closure of the convex hull of the set $\{u\alpha_t(u^*)\mid u\in \cU(Dp)\}$. By uniqueness, we have $a_t=ua_t\alpha_t(u^*)$ for all $u\in\cU(Dp)$. Moreover, by \eqref{eq:Retau} we get $\Re\tau(a_t)\geq \frac{\tau(p)}{2}>0$, hence $a_t\neq 0$.\hfill$\square$
	
	{\bf Claim 2.} Let $t>0$ and $a_t\in p\tilde M\alpha_t(p)$ be a non-zero element such that $a_t = ua_t\alpha_t(u^*)$ for all $u\in\cU(Dp)$.
	Then there exists $b\in Q$ such that $a_{2t}:=\alpha_t(\beta(a_t^*)ba_t)\neq 0$. Moreover, $a_{2t}\in p\tilde M\alpha_{2t}(p)$ satisfies $a_{2t} = ua_{2t}\alpha_{2t}(u^*)$ for all $u\in\cU(Dp)$.
	
	{\it Proof of Claim 2.} To prove the first part of the claim, assume that $\alpha_t(\beta(a_t^*)ba_t)=0$ and thus $\beta(a_t^*)ba_t=0$, for all $b\in Q$.
Thus, if we let $r=a_ta_t^*\in\tilde M$, then since $\beta(u_1^*)=u_1^*$, we get that \begin{equation}\label{a_t}\text{$\beta(u_1ru_1^*)u_2ru_2^*=\beta(u_1a_t)(\beta(a_t^*)u_1^*u_2a_t)(a_t^*u_2^*)=0$, for all $u_1,u_2\in\mathcal U(Q)$.}\end{equation}
	Let $s$ be the element of minimal $\|.\|_2$-norm in the $\|.\|_2$-closure of the convex hull of the set $\{uru^*\mid u\in\mathcal U(Q)\}$. 
	Since $\tau(s)=\tau(r)>0$ and $s\geq 0$, we get that $s\not=0$ and further that $s^2\not=0$. 
	By uniqueness, we have that $s\in Q'\cap \tilde M$ and since $Q'\cap \tilde M\subset M$ 
  by Lemma~\ref{lem:QMomega} we conclude that $s\in M$.
  By combining the last two facts we get that $\beta(s)s=s^2\not=0$.
  This however contradicts  \eqref{a_t} which implies that $\beta(s)s=0$. The moreover assertion is now a straightforward calculation.  \hfill$\square$
	
	By Claim 1, its conclusion holds for $t=2^{-k}$ for some $k\in\N$. Using Claim 2 and induction, we then find  $0\not=a_1\in p\tilde M\alpha_1(p)$ such that $a_1=ua_1\alpha_1(u^*)$, for all $u\in\mathcal U(Dp)$. \end{proof}

\begin{proof}[Proof of Theorem~\ref{thm:spgap}]
	Let $p\in (Q'\cap M)'\cap M$ be a non-zero projection. We need to show that $(Q'\cap M)p\emb_M P$. 
	By Lemma~\ref{lem:alpha1impl} we can find $0\not=a_1\in p\tilde M\alpha_1(p)$ such that $xa_1=a_1\alpha_1(x)$ for all $x\in (Q'\cap M)p$. Thus, the $pMp$-bimodule $_{pMp}L^2(\tilde M)_{\alpha_1(pMp)}$ contains a non-zero $(Q'\cap M)p$-central vector. Since this bimodule is contained in a multiple of $pL^2(M)\otimes_P L^2(M)p$ by assumption (2), we get that $pL^2(M)\otimes_P L^2(M)p$ contains a non-zero $(Q'\cap M)p$-central vector. In other words, the $pMp$-bimodule $pL^2(\langle M,e_P\rangle)p$ contains a non-zero $(Q'\cap M)p$-central vector $\xi$.
	Let $\varepsilon>0$ such that $f=1_{[\varepsilon,\infty)}(\xi^*\xi)\not=0$. Then we have that $f\in ((Q'\cap M)p)'\cap p\langle M,e_P\rangle p$. Since $\hat{\tau}(f)\leq \|\xi\|^2/\varepsilon<\infty$,
	 Theorem~\ref{corner} implies that $(Q'\cap M)p\emb_M P$,  thus finishing the proof of the theorem.
\end{proof}

\begin{proof}[Proof of Corollary~\ref{embed}]

In Section~\ref{ssec:def}, we defined an $s$-malleable deformation ($\tilde C\rtimes\Gamma$, $(\alpha_t)_{t\in\mathbb R},\beta)$ of $C\rtimes\Gamma$, where $\tilde C=C\bar{\otimes}C$ or $\tilde C=C*C$, depending on whether $\Gamma\curvearrowright C$ is the Gaussian action or the free Bogoljubov action associated to $\pi$, respectively. By construction, $\alpha_t(\tilde C)=\tilde C$, $\beta(\tilde C)=\tilde C$ and $\alpha_t(u_g)=u_g$, for all $t\in\mathbb R$ and $g\in\Gamma$.
 Recall that $M=(C\bar{\otimes}D)\rtimes\Gamma$ and put $\tilde M=(\tilde C\bar{\otimes}D)\rtimes\Gamma$. 
 We extend $\alpha_t$ and $\beta$ to automorphisms of $\tilde M$ by letting $\alpha_t(x)=\beta(x)=x$, for all $t\in\mathbb R$ and $x\in D$.  Then $(\tilde M,(\alpha_t)_{t\in\mathbb R},\beta)$ is an $s$-malleable deformation of $M$.
In order to derive the conclusion, it remains to verify that conditions (1) and (2) from Theorem~\ref{thm:spgap} are satisfied with $N=D$ and $P=D\cross\Gamma$.

	As in the proof of \cite[Lemma 3.5]{Va10b}, given a unitary representation $\eta:\Gamma\rightarrow\mathcal U(\mathcal K)$, we define $\mathcal K_{\eta}=\mathcal K\otimes L^2(M)$ and endow it with the following $M$-bimodule structure: $$\text{$(au_g)\cdot (\xi\otimes x)\cdot (bu_h)=\eta_g(\xi)\otimes au_gxbu_h$, for all $a,b\in C\bar{\otimes}D$, $g,h\in\Gamma$, $x\in M$, and $\xi\in\mathcal K$.}$$
If $\eta':\Gamma\rightarrow\mathcal U(\mathcal K')$ is another unitary representation of $\Gamma$, then $\mathcal K_{\eta\otimes\eta'}\cong\mathcal K_{\eta}\otimes_M\mathcal K_{\eta'}$, and if $\eta$ is weakly contained in $\eta'$, then $\mathcal K_{\eta}\subset_{\text{weak}}\mathcal K_{\eta'}$.
	
	{\it Case 1.} $\Gamma\actson^{\sigma} (C,\tau)$ is the Gaussian action associated to $\pi$.
	
	Let $\sigma_0:\Gamma\rightarrow\mathcal U(L^2(C)\ominus\mathbb C1)$ be the restriction of the Koopman representation of $\sigma$ to $L^2(C)\ominus\mathbb C1$.
	Since  $\pi^{\otimes k}$ is weakly contained in the left regular representation $\lambda$ of $\Gamma$, the same holds for $\sigma_0^{\otimes k}$ by \cite[Proposition 2.7]{PS09} and \cite[Proposition II.1.15]{Bo14}. Since the $M$-bimodule $L^2(\tilde M)\ominus L^2(M)$ is isomorphic to $\mathcal K_{\sigma_0}$ we conclude that $$(L^2(\tilde M)\ominus L^2(M))^{\otimes_M k}\cong \mathcal K_{{\sigma_0}}^{\otimes_M k}\cong\mathcal K_{\sigma_0^{\otimes k}}\subset_{\text{weak}}\mathcal K_{\lambda}.$$
	Since $C$ is abelian, hence amenable, $\mathcal K_{\lambda}\cong L^2(M)\otimes_{C\bar{\otimes}D}L^2(M)$ is weakly contained in $L^2(M)\bar{\otimes}_DL^2(M)$, proving condition (1). Since $L^2(\tilde M)=\overline{M\alpha_1(M)}^{\|.\|_2}$ and $\tau(x\alpha_1(y))=\tau(xE_{D\rtimes\Gamma}(y))$, for all $x,y\in M$, the $M$-bimodule $_ML^2(\tilde M)_{\alpha_1(M)}$ is isomorphic to $L^2(M)\otimes_{D\rtimes\Gamma}L^2(M)$. Thus, condition (2) also holds.

	{\it Case 2.} $\Gamma\actson^{\rho} (C,\tau)$ is the free Bogoljubov action associated to $\pi$.

	We will denote still by $\rho$ the diagonal product action of $\Gamma$ on $\tilde C\bar{\otimes}D$.
	
	{\bf Claim.} Let $\xi=\xi_1\xi_2...\xi_n\in\tilde C=C\ast C$, where $\xi_1\in 1\ast (C\ominus\mathbb C1),\xi_2\in (C\ominus\mathbb C1)\ast 1,...,\xi_n\in 1\ast(C\ominus\mathbb C1)$. Then the $M$-bimodule $\mathcal L_{\xi}:=\overline{M\xi M}$ satisfies $\mathcal L_{\xi}^{\otimes_M k}\subset_{\text{weak}}L^2(M)\otimes_D L^2(M)$.
	
	{\it Proof of the claim.} Define $\varphi:\Gamma\rightarrow\mathbb C$ and the completely positive map $\Phi:M\rightarrow M$ by letting $\varphi(g)=\langle\rho_g(\xi),\xi\rangle$ and $\Phi((c\otimes d)u_g)=\tau(c)\varphi(g) (1\otimes d)u_g$, for all $c\in C,d\in D$ and $g\in\Gamma$.
	
	If $c,c'\in C\ast 1$, $d,d'\in D$ and $g,g'\in\Gamma$, then  $\langle c\rho_g(\xi)c',\xi\rangle=\tau(\xi^*c\rho_g(\xi)c')=\tau(c)\tau(c')\varphi(g),$ 
	and thus
	\begin{align*}
		\langle (c\ot d)u_g\xi u_{g'}(c'\ot d'),\xi\rangle &= \delta_{gg',e}\langle c\rho_g(\xi)c',\xi\rangle \langle dd',1\rangle\\
		&= \delta_{gg',e}\varphi(g)\tau(c)\tau(c') \tau(dd')\\
		&= \tau(\Phi((c\ot d)u_g)u_{g'}(c'\ot d')),
	\end{align*}
		In other words, using the notation from section~\ref{ssec:Hilb}, this means that $\mathcal L_{\xi}\cong \cH_{\Phi}$, as $M$-bimodules.
	Note that if $v\in\mathcal U(C), w\in\mathcal U(D), h\in\Gamma$, then for all $d\in D$ and $g\in\Gamma$ we have that
	\begin{equation}\label{Phi}
	[\Phi\circ\text{Ad}((v\otimes w)u_h)]((1\otimes d)u_g)=\tau(v\rho_{hgh^{-1}}(v)^*)\varphi(hgh^{-1})\text{Ad}((1\otimes w)u_h)((1\otimes d)u_g).
	\end{equation}

	Let $\mathcal U$ be the set of unitaries  $u\in M$ of the form $u=(v\otimes w)u_h$, with $v\in\mathcal U(C), w\in\mathcal U(D), h\in\Gamma$. 
	 Since the span of $\mathcal U$ is $\|.\|_2$-dense in $M$, Lemma \ref{compose}(2) implies that the $M$-bimodule $\mathcal L_{\xi}^{\otimes_M k}\cong\mathcal H_{\Phi}^{\otimes_M k}$ is isomorphic to a sub-bimodule of $$\bigoplus_{u_1,...,u_{k-1}\in\; \mathcal U}\mathcal H_{\Phi\circ\text{Ad}(u_{k-1})\circ\Phi\circ\dots\circ\text{Ad}(u_{1})\circ\Phi}.$$
	
	We fix $u_1,...,u_{k-1}\in\mathcal U$ and denote  $\Psi:=\Phi\circ\text{Ad}(u_{k-1})\circ\Phi\circ\dots\circ\text{Ad}(u_{1})\circ\Phi:M\rightarrow M$. 
Thus, in order to prove the claim it suffices to argue that	$\mathcal H_{\Psi}\subset_{\text{weak}}L^2(M)\otimes_D L^2(M)$.
	 To this end, for  $i\in\{1,...,k-1\}$, write $u_i=(v_i\otimes w_i)u_{h_i}$, where $v_i\in\mathcal U(C), w_i\in\mathcal U(D)$ and $h_i\in\Gamma$. We define $U=(1\otimes w_{k-1})u_{h_{k-1}}...(1\otimes w_1)u_{h_1}\in\mathcal U(D\rtimes\Gamma)$ and a positive definite function $\psi:\Gamma\rightarrow\mathbb C$ by letting $$\text{$\psi(g)=\prod_{i=1}^{k-1}\tau(v_i\rho_{h_i...h_1gh_1^{-1}...h_i^{-1}}(v_i)^*)$, for all $g\in\Gamma$.}$$
	By using \eqref{Phi} and induction, it follows that for all $c\in C$, $d\in D$ and $g\in\Gamma$ we have that \begin{equation}\label{psi}\text{$\Psi((c\otimes d)u_g)=\tau(c)\psi(g)\varphi(g)\prod_{i=1}^{k-1}\varphi(h_i...h_1gh_1^{-1}...h_i^{-1})\text{Ad}(U)((1\otimes d)u_g)$}\end{equation}
	Let $\Theta:M\rightarrow  M$ and $\Omega:M\rightarrow M$ be the completely positive maps given by $\Theta(xu_g)=\psi(g)xu_g$ and $\Omega(xu_g)=\varphi(g)\prod_{i=1}^{k-1}\varphi(h_i...h_1gh_1^{-1}...h_i^{-1})xu_g$, for all $x\in C\bar{\otimes}D$ and $g\in\Gamma$. Then \eqref{psi} rewrites as $\Psi=\text{Ad}(U)\circ\Theta\circ\Omega\circ E_{D\rtimes\Gamma}$. By Lemma \ref{compose}(1) we get that \begin{equation}\label{inclusion}\text{the $M$-bimodule $\mathcal H_{\Psi}$ is isomorphic to a sub-bimodule of $\mathcal H_{E_{D\rtimes\Gamma}}\otimes_M\mathcal H_{\Omega}\otimes_M\mathcal H_{\Theta}$.}\end{equation}
	
		 Let $\rho_0:\Gamma\rightarrow\mathcal U(L^2(C)\ominus\mathbb C1)$ be the restriction of the Koopman representation of $\rho$ to $L^2(C)\ominus\mathbb C1$.
	  Since  $\varphi(g)=\langle\rho_g(\xi),\xi\rangle=\prod_{i=1}^n\langle\rho_g(\xi_i),\xi_i\rangle$ and $\xi_i\in C\ominus\mathbb C1$, for all $g\in\Gamma$ and $i\in\{1,...,n\}$, it follows that the $M$-bimodule $\mathcal H_{\Omega}$ is isomorphic to a sub-bimodule of $\mathcal K_{\rho_0^{\otimes kn}}$.
	  Since $\pi^{\otimes k}$ is weakly contained in the left regular representation $\lambda$, so is $\rho_0^{\otimes k}$ by Lemma \ref{weak}. Thus, $\rho_0^{\otimes kn}$ is weakly contained in $\lambda$. Hence, $\mathcal K_{\rho_0^{\otimes kn}}\subset_{\text{weak}}\mathcal K_{\lambda}\cong \mathcal H_{E_{C\bar{\otimes}D}}$. Altogether, we conclude that $\mathcal H_{\Omega}\subset_{\text{weak}}\mathcal H_{E_{C\bar{\otimes}D}}$.  In combination with \eqref{inclusion}, we derive that \begin{equation}\label{weakinc}\mathcal H_{\Psi}\subset_{\text{weak}}\mathcal H_{E_{D\rtimes\Gamma}}\otimes_M\mathcal H_{E_{C\bar{\otimes}D}}\otimes_M\mathcal H_{\Theta}.\end{equation}
	  
	  Since $\mathcal H_{E_N}\cong L^2(M)\otimes_NL^2(M)$, for any von Neumann subalgebra $N\subset M$, and the $(D\rtimes\Gamma)$-$(C\bar{\otimes}D)$-bimodule $L^2(M)$ is isomorphic to $L^2(D\rtimes\Gamma)\otimes_{D}L^2(C\bar{\otimes}D)$, it follows that $\mathcal H_{\Psi}\subset_{\text{weak}}L^2(M)\otimes_D\mathcal H_{\Theta}$.
	  Using that $D$ is regular in $M$ and $\Theta_{|D}=\text{id}_D$, it is easy to see that $L^2(M)\otimes_D\mathcal H_{\Theta}$ is isomorphic to a sub-bimodule of a multiple of $L^2(M)\otimes_DL^2(M)$. Thus, $\mathcal H_{\Psi}\subset_{\text{weak}}L^2(M)\otimes_D L^2(M)$, which finishes the proof of the claim.
	  \hfill$\square$

Since $L^2(\tilde M)\ominus L^2(M)$ decomposes as a direct sum of $M$-bimodules of the form $\mathcal L_{\xi}$ as in the claim, condition (1) follows.
 To verify condition (2), let $\xi\in \tilde C$  be a non-zero element of the form $\xi=\xi_1\xi_2...\xi_n$, where $\xi_1\in 1\ast (C\ominus\mathbb C1),\xi_2\in (C\ominus\mathbb C1)\ast 1,...,\xi_n\in (C\ominus\mathbb C1)\ast 1$. Using a calculation similar to the one in the claim, it follows that the $M$-bimodule $_M\overline{M\xi\alpha_1(M)}_{\alpha_1(M)}$ is isomorphic to a submodule of a multiple of $L^2(M)\otimes_{D\rtimes\Gamma}L^2(M)$. This implies that condition (2) holds in case (2) and finishes the proof of Corollary \ref{embed}.
\end{proof}

\section{Proofs of Theorems \ref{A} and \ref{B}}

The proofs of Theorems \ref{A} and \ref{B} rely on the following consequence of Corollary \ref{embed}.

\begin{lemma}\label{technical}
Let $\Gamma$ be a non-amenable group. For $k\in\mathbb N$, let $\pi_k:\Gamma\rightarrow\mathcal O(\mathcal H_k)$ be an orthogonal representation such that $\pi_k^{\otimes l(k)}$ is weakly contained in the left regular representation of $\Gamma$, for some $l(k)\in\mathbb N$. Let $\Gamma\curvearrowright (B_k,\tau_k)$ be either the Gaussian or the free Bogoljubov action associated to $\pi_k$. Let $\Gamma\curvearrowright (B,\tau):=\bar{\otimes}_{k}(B_k,\tau_k)$ be the diagonal product action and denote $M=B\rtimes\Gamma$.
Let $(M_n)_{n\in\mathbb N}$ be a sequence of von Neumann subalgebras  of $M$ such that $\|x-E_{M_n}(x)\|_2\rightarrow 0$, for every $x\in M$.

Then there exist projections  $p_n\in\mathcal Z(M_n'\cap M)$, for $n\in\mathbb N$, such that  $\lim_{n\rightarrow\infty}\tau(p_n)=1$ and $(M_n'\cap M)p_n\prec_{M}^s(\bar{\otimes}_{k>N}B_k)\rtimes\Gamma$, for every $n,N\in\mathbb N$.

Moreover, if $\Gamma$ is not inner amenable, then there exist projections $r_n\in\mathcal Z(M_n'\cap M)$, for $n\in\mathbb N$, such that $\lim_{n\rightarrow\infty}\tau(r_n)=1$ and $(M_n'\cap M)r_n$ is amenable, for every $n\in\mathbb N$.
\end{lemma}

{\it Proof.}
Let $q_n\in\mathcal Z(M_n'\cap M)$ be the largest projection such that $M_nq_n$ is amenable relative to $B$.
We claim that $\tau(q_n)\rightarrow 0$. Otherwise, after replacing $(M_n)_{n\in\N}$ with a subsequence, we may assume that $\tau(q_n)\rightarrow c>0$. By Lemma \ref{union}, this implies that there is a non-zero projection $q\in\mathcal Z(M)$ such that $Mq$ is amenable relative to $B$. Since $M$ is a factor, this would give that $M$ is amenable relative to $B$ and hence that $\Gamma$ is amenable by \cite[Proposition 2.4]{OP07}, which is a contradiction.
 
 Next,  fix $n\in\mathbb N$ and put $p_n=1-q_n$. Then $M_np'$ is not amenable relative to $B$, for any non-zero projection $p'\in (M_n'\cap M)p_n$. Otherwise, \cite[Lemma 2.6(2)]{DHI16} would provide a non-zero projection $z\in\mathcal Z(M_n'\cap M)p_n$ such that $M_nz$ is amenable relative to $B$, contradicting the maximality of $q_n$. Let $i\in\mathbb N$ and denote $C_i=\bar{\otimes}_{k\not=i}B_k$.
Since $C_i\subset B$, $M_np'$ is not amenable relative to $C_i$, for any non-zero projection $p'\in (M_n'\cap M)p_n$.  
Since $\Gamma\curvearrowright B_i$ is either the Gaussian or the free Bogoljubov action associated to $\pi_i$, and a multiple of $\pi_i$ is weakly contained in the left regular representation of $\Gamma$, we can apply Corollary \ref{embed}  to the inclusion $M_np_n\subset M=(B_i\bar{\otimes}C_i)\rtimes\Gamma$ to deduce that
\begin{equation}\label{C_i}\text{$(M_n'\cap M)p_n\prec_{M}^sC_i\rtimes\Gamma$, for all $i\in\mathbb N$. }\end{equation}

Let $N\in\mathbb N$. Since the subalgebras $\{C_i\}_{i=1}^N$ of $M$ are regular and any two form a commuting square, \eqref{C_i} and \cite[Lemma 2.8(2)]{DHI16} imply that $(M_n'\cap M)p_n\prec_{M}^s\cap_{i=1}^N(C_i\rtimes\Gamma)=(\bar{\otimes}_{k>N}B_k)\rtimes\Gamma$. Since $\tau(p_n)\rightarrow 1$, this proves the main assertion.

For the moreover assertion, assume that $\Gamma$ is not inner amenable.  Then by \cite{Ch82} we get that $M'\cap M^{\omega}\subset B^{\omega}$ and hence $\prod_{\omega}(M_n'\cap M)\subset M'\cap M^{\omega}\subset B^{\omega}.$  By combining this with  \cite[Lemmas 2.2 and 2.3]{IS18} we can find projections $q_n\in\mathcal Z(M_n'\cap M)$ such that $\tau(q_n)\rightarrow 1$ and \begin{equation}\label{M_n}
\text{$(M_n'\cap M)q_n\prec_{M}^sB$, for every $n\in\mathbb N$.}
\end{equation}
Put $r_n=p_n\wedge q_n\in\mathcal Z(M_n'\cap M)$. Then $(M_n'\cap M)r_n\prec_{M}^sB$ and $(M_n'\cap M)r_n\prec_{M}^s(\bar{\otimes}_{k>N}B_k)\rtimes\Gamma$, for every $n,N\in\mathbb N$. Since  $B$ is regular in $M$, $B$ and $(\bar{\otimes}_{k>N}B_k)\rtimes\Gamma$ form a commuting square and $B\cap ((\bar{\otimes}_{k>N}B_k)\rtimes\Gamma)=\bar{\otimes}_{k>N}B_k$, \cite[Lemma 2.8(2)]{DHI16} implies that \begin{equation}\label{strong}
\text{$(M_n'\cap M)r_n\prec_M^s\bar{\otimes}_{k>N}B_k$, for every $n,N\in\mathbb N$.}
\end{equation}
For $N\in\mathbb N$, put $Q_N=\bar{\otimes}_{k>N}B_k$ and $R_N=(\bar{\otimes}_{k\leq N}B_k)\rtimes\Gamma$. Then $\|x-E_{R_N}(x)\|_2\rightarrow 0$, for any $x\in M$, and $_{Q_N}L^2(M)_{R_N}\cong$ $_{Q_N}L^2(Q_N)\otimes L^2(R_N)_{R_N}$, for any $N\in\mathbb N$. These facts and \eqref{strong} imply that we can apply Lemma \ref{generalamen} to deduce that $(M_n'\cap M)r_n$ is amenable, for every $n\in\mathbb N$.
\hfill$\blacksquare$

\subsection{Proof of Theorem \ref{A}} Assume by contradiction that $M$ admits a residual sequence $(A_n)_n$. For $n\in\mathbb N$, let $M_n=A_n'\cap M$. 
Since $\prod_{\omega}A_n\subset \cap_nA_n^{\omega}\subset M'\cap M^{\omega}$, Lemma \ref{increase} implies that $\|x-E_{M_n}(x)\|_2\rightarrow 0$, for every $x\in M$.
 By Lemma \ref{technical} we can find projections $p_n\in\mathcal Z(M_n'\cap M)$ such that $\tau(p_n)\rightarrow 1$ and $(M_n'\cap M)p_n\prec_{M}^s(\bar{\otimes}_{l>N}B_l)\rtimes\Gamma$, for every $n,N\in\mathbb N$.
 Since $A_n\subset M_n'\cap M$, we thus get that \begin{equation}\label{A_n}
\text{$A_np_n\prec_{M}^s(\bar{\otimes}_{l>N}B_l)\rtimes\Gamma$, for every $n,N\in\mathbb N$.}
\end{equation}
 
Let $n\in\mathbb N$ be fixed such that $\tau(p_n)>15/16$. Recall that $\Gamma\curvearrowright B_k$ is the Gaussian action associated to $\pi_k$
 and denote 
 $U_k^m=\omega(\xi_k^m)\in\mathcal U(B_k)$, for every $k,m\in\mathbb N$.

{\bf Claim}. There exists $k\in\mathbb N$ such that
$\|U_k^m-E_{A_n}(U_k^m)\|_2\leq 1/16$, for every $m\in\mathbb N$.

{\it Proof of the claim.}
Assuming the claim is false,  for every $k\in\mathbb N$, we can find $m(k)\in\mathbb N$ such that $U_k:=U_k^{m(k)}\in\mathcal U(B_k)$ satisfies $\|U_k-E_{A_n}(U_k)\|_2>1/16$.
Since $1-e^{-t}\leq t$, for any $t\geq 0$, we get \begin{align*}\|u_gU_ku_g^*-U_k\|_2&=\|\omega(\pi_k(g)(\xi_k^{m(k)}))-\omega(\xi_k^{m(k)})\|_2\\&=\sqrt{2(1-\exp(-\|\pi_k(g)(\xi_k^{m(k)})-\xi_k^{m(k)}\|^2))}\\&\leq\sqrt{2}\|\pi_k(g)(\xi_k^{m(k)})-\xi_k^{m(k)}\|, \;\text{for every $g\in\Gamma$.}\end{align*}
Since $\sup_{m\in\mathbb N}\|\pi_k(g)(\xi_k^m)-\xi_k^m\|\rightarrow 0$, we deduce that $\|u_gU_ku_g^*-U_k\|_2\rightarrow 0$, for every $g\in\Gamma$. Since $U_k\in\mathcal U(B_k)$, we also have that $U_kx=xU_k$, for every $x\in B$. By combining the last two facts we get that $U:=(U_k)\in M'\cap M^{\omega}$. However, since $\|U-E_{A_n^{\omega}}(U)\|_2=\lim_{k\rightarrow\omega}\|U_k-E_{A_n}(U_k)\|_2\geq 1/16$, this contradicts that $M'\cap M^{\omega}\subset A_n^{\omega}$. Altogether, this proves the claim. \hfill$\square$

Let $k\in\mathbb N$ be as in the claim and put $V_m=U_k^m-\tau(U_k^m)$. Then we have $V_m\in B_k$, $\|V_m\|\leq 2$, $\|V_m\|_2=\sqrt{1-\exp(-1)}$ and $\|V_m-E_{A_n}(V_m)\|_2\leq 1/16$, for every $m\in\mathbb N$. Since $\tau(V_{m'}^*V_m)=0$, for all $m\not=m'$, we also have that $V_m\rightarrow 0$ weakly.

By specializing \eqref{A_n} to $N=k$ we get that $A_np_n\prec_{M}^s(\bar{\otimes}_{l>k}B_l)\rtimes\Gamma$. This implies that we can find a finite dimensional subspace $\mathcal K\subset\bar{\otimes}_{l\leq k}B_l$ such that if $e$ denotes the orthogonal projection from $L^2(M)$ onto the $\|.\|_2$-closed linear span of $\{(y\otimes z)u_g\mid y\in\mathcal K,z\in \bar{\otimes}_{l>k}B_l,g\in\Gamma\}$, then
\begin{equation}\label{e}
\text{$\|x-e(x)\|_2\leq 1/16$, for all $x\in (A_np_n)_1$.}
\end{equation}
Next, if $m\in\mathbb N$, then $\|V_m-E_{A_n}(V_m)\|\leq 1/16$ and hence $\|V_mp_n-E_{A_n}(V_m)p_n\|_2\leq 1/16$.
Since $E_{A_n}(V_m)p_n\in A_np_n$ and $\|E_{A_n}(V_m)p_n\|\leq 2$, \eqref{e} gives $\|E_{A_n}(V_m)p_n-e(E_{A_n}(V_m)p_n)\|_2\leq 1/8$. Combining the last two inequalities further implies that \begin{equation}\label{V_m}\text{$\|V_mp_n-e(V_mp_n)\|_2\leq 1/4$, for every $m\in\mathbb N$.}\end{equation}

Now, we claim that \begin{equation}\label{eq:weak}\text{$\lim_{m\rightarrow\infty}\|E_{(\bar{\otimes}_{l>k}B_l)\rtimes\Gamma}(xV_my)\|_2=0$, for all $x,y\in M$.}\end{equation} Indeed, it is enough to check this when $x=u_g(a\otimes b)$ and $y=(c\otimes d)u_h$, for $a,c\in\bar{\otimes}_{l\leq k}B_l$, $b,d\in \bar{\otimes}_{l>k}B_l$ and $g,h\in\Gamma$. Then, since $V_m\in B_k$, we have  $E_{(\bar{\otimes}_{l>k}B_l)\rtimes\Gamma}(xV_my)=\tau(aV_mb)u_gbdu_h$ and the conclusion follows since $V_m\rightarrow 0$ weakly. This proves \eqref{eq:weak}.

Let $\{\xi_j\}_{j=1}^r$ be an orthonormal basis for $\mathcal K$. Since $E_{(\bar{\otimes}_{l>k}B_l)\rtimes\Gamma}(\xi_i^*\xi_j)=\delta_{i,j}$, for all $i,j\in\{1,...,r\}$, we get that $e(x)=\sum_{j=1}^r\xi_jE_{(\bar{\otimes}_{l>k}B_l)\rtimes\Gamma}(\xi_j^*x)$, for every $x\in M$. In combination with \eqref{eq:weak} it follows that $\|e(V_mp_n)\|_2\rightarrow 0$. On the other hand,  since $\|V_m\|\leq 2$ and $\tau(p_n)>15/16$, we have that \begin{align*}\|V_mp_n\|_2\geq \|V_m\|_2-\|V_m(1-p_n)\|_2&\geq \|V_m\|_2-2\|1-p_n\|_2\\&=\sqrt{1-\exp(-1)}-2\sqrt{1-\tau(p_n)}>1/4,\;\;\text{for every $m\in\mathbb N$}.\end{align*}
Altogether, we get that $\liminf_{m\rightarrow\infty} \|V_mp_n-e(V_mp_n)\|_2>1/4$, which contradicts \eqref{V_m}. So $M$ cannot have a residual sequence.
\qed

\begin{remark}
The proof of Theorem \ref{A}  shows that there is no sequence $(A_n)_{n\in\mathbb N}$ of von Neumann subalgebras of $M$ such that $\prod_{\omega}A_n\subset M'\cap M^{\omega}\subset\cap_{n\in\mathbb N}A_n^{\omega}$. In particular, there is no  sequence $(A_n)_{n\in\mathbb N}$ of von Neumann subalgebras of $M$ which satisfies conditions  (2) and (3) of Definition \ref{mdef}.
\end{remark}

\subsection{Proof of Theorem \ref{B}} Recall that $\Gamma\curvearrowright B_k$ is the free Bogoljubov action associated to $\pi_k$ and denote $W_{k,m}=W(\xi_k^m)\in B_k$, for $k\in\mathbb N$ and $m\in\{1,2\}$. Then for any $k\in\mathbb N$, $\{W_{k,1},W_{k,2}\}$ are freely independent semicircular operators with $\|W_{k,1}\|=\|W_{k,2}\|=2$. Moreover, if $m\in\{1,2\}$, then for any $g\in\Gamma$ we have that $
\|u_gW_{k,m}u_g^*-W_{k,m}\|_2=\|W(\pi_k(g)(\xi_k^m))-W(\xi_k^m)\|_2=\|\pi_k(g)(\xi_k^m)-\xi_k^m\|\rightarrow 0$. Since $W_{k,m}\in B_k$, we also have that $\|W_{k,m}x-xW_{k,m}\|_2\rightarrow 0$, for every $x\in B$. By combining the last two facts, we get that $W_m=(W_{k,m})_k\in M'\cap M^{\omega}$. 

Let us first prove the moreover assertion. To this end, let $P\subset M'\cap M^{\omega}$ be the von Neumann subalgebra generated by $W_1$ and $W_2$. Assume by contradiction that there is a sequence $(A_n)_n$ of von Neumann subalgebras of $M$ such that $$P\subset\prod_{\omega}A_n\subset M'\cap M^{\omega}.$$ 

 For $n\in\mathbb N$, let $M_n=A_n'\cap M$. 
Lemma \ref{increase} implies that $\lim_{n\rightarrow\omega}\|x-E_{M_n}(x)\|_2\rightarrow 0$, for every $x\in M$.
The moreover assertion of Lemma \ref{technical}  implies the existence of projections $r_n\in\mathcal Z(M_n'\cap M)$ such that $\lim_{n\rightarrow\omega}\tau(r_n)\rightarrow 1$ and $(M_n'\cap M)r_n$ is amenable, for every $n\in\mathbb N$. Thus, $A_nr_n$ is amenable, for every $n\in\mathbb N$.

If $n\in\mathbb N$, then since $W_m=(W_{k,m})_k\in P\subset\prod_{\omega}A_k$, there is $k_n\in\mathbb N$ satisfying $\tau(r_{k_n})\geq 1-1/n^2$ and $\|W_{k_n,m}-E_{A_{k_n}}(W_{k_n,m})\|_2\leq 1/n$, for every $m\in\{1,2\}$. Thus, if  $B_n=A_{k_n}r_{k_n}\oplus\mathbb C(1-r_{k_n})$, then \begin{equation}\label{W_n}\text{$\|W_{k_n,m}-E_{B_n}(W_{k_n,m})\|_2\leq 1/n+\|1-r_{k_n}\|_2\leq 2/n$, for every $n\in\mathbb N$ and $m\in\{1,2\}$.} \end{equation}

Let $N$ be the II$_1$ factor generated by two freely independent semicircular operators $S_1,S_2$ with $\|S_1\|=\|S_2\|=2$.
For $n\in\mathbb N$, let $\pi_n:N\rightarrow M$ be the unique trace preserving $*$-homomorphism such that $\pi_n(S_m)=W_{k_n,m}$, for all $m\in\{1,2\}$.
Then \eqref{W_n} gives that $\|\pi_n(x)-E_{B_n}(\pi_n(x))\|_2\rightarrow 0$, for every $x\in N$. Since $B_n$ is amenable, for every $n\in\mathbb N$, Corollary \ref{amenable} implies that $N$ is amenable. Since $N\cong L(\mathbb F_2)$ is not amenable, this gives a contradiction and thus proves the moreover assertion.

To prove the main assertion, assume by contradiction that $M$ admits a residual sequence $(A_n)_n$. Then $P\subset M'\cap M^{\omega}=\cap_nA_n^{\omega}$ and since $P$ is separable, we can find an increasing sequence of positive integers $(k_n)$ such that $P\subset \prod_{\omega}A_{k_n}$. Since $\prod_{\omega}A_{n_k}\subset\cap_nA_n^{\omega}=M'\cap M^{\omega}$, this contradicts the moreover assertion.
\hfill$\blacksquare$

\section{Stability}

\subsection{Proof of Proposition~\ref{C}}
Since $P$ is amenable, it  is aproximately finite dimensional by Connes' theorem \cite{Co76}. Thus, we can find  an increasing sequence $(B_k)_k$ of finite dimensional von Neumann subalgebras such that $P=(\cup_k B_k)''$. If $k\in\N$, then since $B_k$ is finite dimensional, there exists $S_k\in\omega$ such that for every $n\in S_k$ we have an embedding $B_k\subset M_n$ in such a way that the embedding $B_k\subset\prod_\omega M_n$ is the diagonal embedding. Put $S_0=\N$.

{\bf Claim.} There exists a sequence $(k_n)\subset\mathbb N$ such that $n\in S_{k_n}$, for all $n\in\mathbb N$, $\lim_{n\rightarrow\omega}k_n=+\infty$, and $$Q\subset \prod_\omega (B_{k_n}'\cap M_n).$$
{\it Proof of the claim.} Since $B_k$ is finite dimensional, $Q\subset P'\cap \prod_\omega M_n \subset B_k'\cap \prod_\omega M_n = \prod_\omega (B_k'\cap M_n)$, for every $k\in\mathbb N$. Hence $Q\subset \cap_{k\in\N} \prod_\omega (B_k'\cap M_n)$, i.e.
\begin{equation}\label{eq:q}
\text{$\lim_{n\rightarrow\omega} \norm{q_n - E_{B_k'\cap M_n}(q_n)}_2 = 0$, for all $k\in\mathbb N$ and $q=(q_n)\in Q.$}
\end{equation}
Now, let $\{q^{(m)}\}_{m\in\N}$ be a $\norm{.}_2$-dense sequence in $(Q)_1$. Let $X_0=\N$ and 
\[
X_k = \left\{ n\in S_k\mid \norm{q^{(i)}_n - E_{B_k'\cap M_n}(q^{(i)}_n)}_2\leq \frac{1}{k}, \text{ for all } 1\leq i\leq k\right\}.
\]
For $n\in\mathbb N$, define $k_n$ to be the largest $k\leq n$ such that $n\in X_k$. We claim that $\lim_{n\rightarrow\omega}k_n=+\infty$. Otherwise, there exists $k\in\N$ such that $\{n\in\N\mid k_n=k\}\in\omega$. Then $\{n\in\N\mid n\not\in X_{k+1}\}\in\omega$.
Since $S_{k+1}\in\omega$, this would imply the existence of $i\in\{1,\dots,k+1\}$ such that we have $$\left\{n\in\mathbb N\mid\norm{q^{(i)}_n - E_{B_{k+1}'\cap M_n}(q^{(i)}_n)}_2> \frac{1}{k+1}\right\}\in\omega,$$ and thus

\[
\lim_{n\rightarrow\omega} \norm{q^{(i)}_n - E_{B_{k+1}'\cap M_n}(q^{(i)}_n)}_2\geq \frac{1}{k+1},
\]
contradicting \eqref{eq:q}. By construction $Q\subset \prod_\omega (B_{k_n}'\cap M_n)$, which finishes the proof of the claim.\hfill$\square$

Taking $(k_n)$ as in the Claim, we also have that $P\subset\prod_{\omega} B_{k_n}$.
Thus, $P_n=B_{k_n}$ and $Q_n=B_{k_n}'\cap M_n$ verify the conclusion of Proposition~\ref{C}.
\qed

\subsection{Proof of Theorem~\ref{E}} In the proof of Theorem~\ref{E} we will need the following consequence of Corollary~\ref{embed}.
Recall that a tracial von Neumann algebra $(M,\tau)$ is called \textit{solid} \cite{Oz03} if the relative commutant $P'\cap M$ is amenable,  for any diffuse von Neumann subalgebra $P\subset M$. 

\begin{lemma}\label{lem:solid}
	Let $\Gamma$ be a countable group and $\pi:\Gamma\rightarrow\mathcal O(\mathcal H_{\mathbb R})$ be a mixing orthogonal representation. Assume that $\pi^{\otimes k}$ is weakly contained in the left regular representation of $\Gamma$, for some $k\in\mathbb N$. Let $\Gamma\curvearrowright (C,\tau)$ be 
	the free Bogoljubov action associated to $\pi$.
	If $L(\Gamma)$ is solid, then $C\cross \Gamma$ is solid.
\end{lemma}
\begin{proof}
Assume that $L(\Gamma)$ is solid. In order to prove that $M=C\rtimes\Gamma$ is solid it suffices to show that if  $P\subset M$ is a diffuse von Neumann subalgebra, then $P'\cap M$ has an amenable direct summand. Suppose by contradiction that $P'\cap M$ has no amenable direct summand. 
 By applying Corollary~\ref{embed}, we get that $P\emb_M L(\Gamma)$. Hence there exist projections $p\in P, q\in L(\Gamma)$, a $*$-homomorphism $\theta: pPp\rightarrow qL(\Gamma) q$, and a non-zero partial isometry $v\in qMp$ such that $\theta(x)v=vx$ for all $x\in pPp$. 
Since $\pi$ is mixing, the action $\Gamma\actson C$ is mixing by \cite[Proposition 2.6]{Ho12b}. Since $\theta(pPp)\subset qL(\Gamma)q$ is a diffuse subalgebra and $vv^*\in\theta(pPp)'\cap qMq$, \cite[Theorem 3.1]{Po03} implies that $q_0:=vv^*\in L(\Gamma)$.  
Thus, $P_0:=vPv^*$ is a diffuse subalgebra of $q_0L(\Gamma) q_0$.  Since $v(P'\cap M)v^*\subset q_0Mq_0$ is a subalgebra which commutes with $P_0$,  \cite[Theorem 3.1]{Po03}  gives that $v(P'\cap M)v^*\subset P_0'\cap q_0L(\Gamma)q_0$. Since $L(\Gamma)$ is solid, we get that $v(P'\cap M)v^*$ 
 is amenable and thus $P'\cap M$ has an amenable direct summand. This finishes the proof of the lemma.
 \end{proof}

\begin{proof}[Proof of Theorem~\ref{E}] 
First, note that if $W$ is a self-adjoint operator in a tracial von Neumann algebra whose distribution with respect to the trace is the semicircular law supported on $[-2,2]$, then  $\{W\}''$ is a diffuse abelian von Neumann algebra. Hence we can find a Borel function $f:[-2,2]\rightarrow\mathbb T$ such that $U=f(W)\in\{W\}''$ is a Haar unitary, i.e. $\tau(U^n)=0$, for all $n\in\mathbb Z\setminus\{0\}$. From now on, fix two freely independent self-adjoint operators $W_1,W_2$
in a tracial von Neumann algebra whose distribution is the semicircular law supported on $[-2,2]$. Define $U_1= f(W_1)$ and $U_2=f(W_2)$. Then $U_1$ and $U_2$ are freely independent Haar unitaries and thus $N=\{U_1,U_2\}''$ satisfies  $N=\{U_1\}''*\{U_2\}''\cong L(\mathbb F_2)$. 

Let $\Gamma=\mathbb F_2$ and $a_1,a_2\in\Gamma$ be free generators.
Let $\pi_k:\Gamma\rightarrow\cO(\mathcal H_k)$, $k\in\mathbb N$, be a sequence of mixing representations such that a tensor multiple of $\pi_k$ is weakly contained in the left regular representation of $\Gamma$, and there exist unit vectors $\xi_{k,m}\in\mathcal H_k$ such that $\|\pi_k(g)(\xi_k^m)-\xi_k^m\|\rightarrow 0$, for every $m\in\{1,2\}$ and $g\in\Gamma$. For instance, let $(\pi_k)_{k\in\N}$ be as in Example \ref{free} and notice that by construction $\pi_k$ is indeed mixing, for every $k\in\N$.
Let $\Gamma\actson B_k$ be the free Bogoljubov action associated to $\pi_k$ and denote $M_k=B_k\rtimes\Gamma$, for every $k\in\mathbb N$.

  Then $W_{k,m}=W(\xi_k^m)\in B_k$ is a self-adjoint operator whose distribution is the semicircular law supported on $[-2,2]$.  Moreover, $\|u_g W_{k,m}- W_{k,m}u_g\|_2=\|\pi_k(g)(\xi_k^m)-\xi_k^m\|\rightarrow 0$, for every $m\in\{1,2\}$ and $g\in\Gamma$. 
Thus, if we put $U_{k,m}=f(W_{k,m})\in\mathcal U(B_k)$, then \begin{equation}\label{comm}\text{$\|u_gU_{k,m}-U_{k,m}u_g\|_2\rightarrow 0$, for every $m\in\{1,2\}$ and $g\in\Gamma$.}\end{equation}

 Let $\rho_k:N\rightarrow M_k$ be the unique trace preserving $*$-homomorphism given by $\rho_k(U_1)=U_{k,1}$ and $\rho_k(U_2)=U_{k,2}$. Then \eqref{comm} rewrites as \begin{equation}\label{comm2}\text{$\|u_g\rho_k(U_m)-\rho_k(U_m)u_g\|_2\rightarrow 0$, for every $m\in\{1,2\}$ and $g\in\Gamma$.}\end{equation}

In the rest of the proof, we treat the two assertions of Theorem~\ref{E} separately.

{\it Part 1.} We first prove that $\Gamma\times\Gamma$ is not W$^*$-tracially stable. This readily implies that $\mathbb F_l\times\mathbb F_m$ is not W$^*$-tracially stable, for every $2\leq l,m\leq +\infty$.  Assume by contradiction that $\Gamma\times\Gamma$ is W$^*$-tracially stable.
Using \eqref{comm2} we can define a homomorphism $\varphi:\Gamma\times\Gamma\rightarrow\mathcal U(\prod_{\omega}M_k)$
by letting \begin{equation}\label{varphi}\text{$\varphi(a_m,e)=(\rho_k(U_m))_k$ and $\varphi(e,g)=u_g$, for all $m\in\{1,2\}$ and $g\in\Gamma$.}\end{equation}

Since $\Gamma\times\Gamma$ is assumed W$^*$-tracially stable, there must be homomorphisms $\varphi_k:\Gamma\times\Gamma\rightarrow\mathcal U(M_k)$ such that $\varphi=(\varphi_k)_k$. Let $C_k=\varphi_k(\Gamma\times\{e\})''$ and $D_k=\varphi_k(\{e\}\times\Gamma)''$. Then $C_k$ and $D_k$ are commuting von Neumann subalgebras of $M_k$ and  we have that \begin{equation}\label{C_k}\text{$\lim_{k\rightarrow\omega}\|\rho_k(U_m)-E_{C_k}(\rho_k(U_m))\|_2=0$, for every $m\in\{1,2\}$, and} \end{equation}
\begin{equation}\label{D_k}\text{$\lim_{k\rightarrow\omega}\|u_g-E_{D_k}(u_g)\|_2=0$, for every $g\in\Gamma$.}\end{equation}

Then \eqref{C_k} implies that $\lim\limits_{k\rightarrow\omega}\|\rho_k(x)-E_{C_k}(\rho_k(x))\|_2\rightarrow 0$, for every $x\in N$.
Since $N$ is a non-amenable II$_1$ factor,  Corollary \ref{amenable} implies that if $p_k\in\mathcal Z(C_k)$ is the largest projection such that $C_kp_k$ is amenable, then $\lim_{k\rightarrow\omega}\tau(p_k)=0$. 
Since $L(\Gamma)$ is also a non-amenable II$_1$ factor, by repeating this argument using $\eqref{D_k}$, it follows that $\lim_{k\rightarrow\omega}\tau(q_k)=0$, where $q_k\in\mathcal Z(D_k)$ denotes the largest projection such that $D_kq_k$ is amenable. 
Thus, for every $k\in\N$, $r_k:=(1-p_k)(1-q_k)\in\{C_k,D_k\}''$ is a projection such that $C_kr_k$ and $D_kr_k$ have no amenable direct summands, and $\lim_{k\rightarrow\omega}\tau(r_k)=1$.
In particular, we can find $k$ such that $r_k\not=0$. This implies that $r_kMr_k$ and thus $M$ is not solid, which is a contradiction by Lemma \ref{lem:solid}. This finishes the proof of the first assertion of Theorem~\ref{E}.

{\it Part 2.} For the moreover assertion, put $B=\bar{\otimes}_{k\in\mathbb N}B_k$ and $M=B\cross\Gamma$. Using the natural embeddings $M_k\subset M$, for every $k\in\mathbb N$, we can view $\prod_{\omega}M_k$ as a subalgebra of $M^{\omega}$. Thus, we may view $\varphi$ as a homomorphism $\varphi:\Gamma\times\Gamma\rightarrow\mathcal U(M^{\omega})$. 
	Since by the definition \eqref{varphi} of $\varphi$ we have $\varphi(a,e)\in B^{\omega}$, $\tau(\varphi(a,e))=\delta_{a,e}$ and $\varphi(e,g)=u_g$, it follows that $\tau(\varphi(a,g))=\tau(\varphi(a,e)u_g)=\delta_{(a,g),(e,e)}$, for all $a,g\in\Gamma$. Thus, $\varphi$ extends to a $*$-homomorphism $\varphi:L(\mathbb F_2\times\mathbb F_2)\rightarrow M^{\omega}$.

	We claim that there are no homomorphisms $\varphi_k:\Gamma\times\Gamma\rightarrow\mathcal U(M)$ such that $\varphi=(\varphi_k)_k$.
Assume by contradiction that such homomorphisms $(\varphi_k)$ exist. Then $C_k=\varphi_k(\Gamma\times\{e\})''$ and $D_k=\varphi_k(\{e\}\times\Gamma)''$ are commuting von Neumann subalgebras of $M$ such that \eqref{C_k} and \eqref{D_k} hold.
	
	 Since $\Gamma$ is non-amenable, \cite[Proposition 2.4]{OP07} implies that $L(\Gamma)$ is not amenable relative to $B$ inside $M$. Thus, since $L(\Gamma)'\cap M=\mathbb C1$, there is no non-zero projection $q\in L(\Gamma)'\cap M$ such that $L(\Gamma)q$ is amenable relative to $B$ inside $M$.
Let $q_k\in D_k'\cap M$ be the largest projection such that $D_kq_k$ is amenable relative to $B$ inside $M$. Then by \cite[Lemma 2.6]{DHI16} we have that $q_k\in\mathcal Z(D_k'\cap M)$. Since by \eqref{D_k} we have that $\lim_{\omega}\|x-E_{D_k}(x)\|_2=0$, for every $x\in L(\Gamma)$, we can apply Lemma \ref{union} to conclude that $\lim_{\omega}\tau(q_k)=0$.
	
Next, fix $k\in\mathbb N$. Then $D_kp'$ is not amenable relative to $B$ inside $M$, for any non-zero projection $p'\in (D_k'\cap M)(1-q_k)$. 
For $i\in\mathbb N$, let $R_i=\bar{\otimes}_{l\not= i}B_l$. Then by applying Corollary \ref{embed} to the decomposition $M=(B_i\bar{\otimes}R_i)\rtimes\Gamma$ it follows that $C_k(1-q_k)\prec_{M}^{s}R_i\rtimes\Gamma$, for every $i\in\mathbb N$. If $N\in\mathbb N$, then the subalgebras $\{R_i\rtimes\Gamma\}_{i=1}^N$ of $M$ are regular and any two form a commuting square. Since $\cap_{i=1}^N(R_i\rtimes\Gamma)=(\bar{\otimes}_{l>N}B_l)\rtimes\Gamma$, \cite[Lemma 2.8(2)]{DHI16} implies that \begin{equation}\label{cornerI}\text{$C_k(1-q_k)\prec_{M}^{s}(\bar{\otimes}_{l>N}B_l)\rtimes\Gamma$, for every $k, N\in\mathbb N$}.
	\end{equation}

Since $\Gamma=\langle a_1,a_2\rangle$ is not inner amenable, we can find a constant $c>0$ such that \begin{equation}\label{inner}\text{$\|x-E_B(x)\|_2\leq c(\|[x,u_{a_1}]\|_2+\|[x,u_{a_2}]\|_2)$, for every $x\in M$.}\end{equation} For $k\in\mathbb N$, denote $\varepsilon_k=\|u_{a_1}-E_{D_k}(u_{a_1})\|_2+\|u_{a_2}-E_{D_k}(u_{a_2})\|_2$. Then \eqref{D_k} implies that $\lim_{\omega}\varepsilon_k=0$. Since $C_k$ and $D_k$ commute, we have that $\|[x,u_{a_1}]\|_2+\|[x,u_{a_2}]\|_2\leq 2\varepsilon_k$, for all $x\in (C_k)_1$. In combination with \eqref{inner}, we get that $\|x-E_B(x)\|_2\leq 2c\varepsilon_k$, for all $x\in (C_k)_1$. By applying \cite[Lemma 2.2]{IS18} we derive the existence of a projection $r_k\in\mathcal Z(C_k'\cap M)$ such that $\tau(r_k)\geq 1-2c\varepsilon_k$ and \begin{equation}\label{cornerII}\text{$C_kr_k\prec_{M}^sB$, for every $k\in\mathbb N$.}\end{equation} 
Since $B$ and $(\bar{\otimes}_{l>N}B_l)\rtimes\Gamma$ are regular subalgebras of $M$ which form a commuting square, if  $p_k=(1-q_k)r_k\in C_k'\cap M$, by combining \eqref{cornerI}, \eqref{cornerII} and \cite[Lemma 2.8(2)]{DHI16} we get that
\begin{equation}
\label{cornerIII}\text{$C_kp_k\prec_M\bar{\otimes}_{l>N}B_l$, for every $k,N\in\mathbb N$.}
\end{equation}
Using \eqref{cornerIII} and reasoning as at the end of the proof of Lemma \ref{technical}, it follows that $C_kp_k$ is amenable, for every $k\in\mathbb N$. 
Since $\lim_{\omega}\tau(q_k)=0$ and $\lim_{\omega}\tau(r_k)=1$, we get that $\lim_{\omega}\tau(p_k)=1$. On the other hand, \eqref{C_k} implies that $\lim_{\omega}\|\rho_k(x)-E_{C_k}(\rho_k(x))\|_2=0$, for every $x\in N$. By applying Corollary \ref{amenable} we derive that $N$ is amenable, which is a contradiction.
\end{proof}


\begin{thebibliography}{ABC99}

  \bibitem[At18]{At18} S. Atkinson: {\it Some results on tracial stability and graph products}, preprint arXiv:1808.04664.

  \bibitem[BCI15]{BCI15} R. Boutonnet, I. Chifan, and A. Ioana: \textit{II$_1$ factors with non-isomorphic ultrapowers}, Duke Math. J. Volume {\bf 166}, Number 11 (2017), 2023-2051. 

  \bibitem[Bo12]{Bo12} R. Boutonnet: {\it On solid ergodicity for Gaussian actions}, J. Funct. Anal. {\bf 263} (2012), no. 4, 1040-1063.
  
  \bibitem[Bo14]{Bo14} R. Boutonnet: {\it Plusieurs aspects de rigidit\'e des alg\`ebres de von Neumann}, PhD Thesis, 2014.
  
  
  \bibitem[Ch69]{Ch69} W.-M. Ching: {\it Non-isomorphic non-hyperfinite factors}, Canad. J. Math. {\bf 21} (1969), 1293-1308.
  
  
  \bibitem[Ch82]{Ch82} M. Choda: {\it Inner amenability and fullness}, Proc. Amer. Math. Soc. {\bf 86} (1982), 663-666.
 
  \bibitem[Co76]{Co76} A. Connes: {\it Classification of injective factors},  Ann. of Math. (2) {\bf 104} (1976), no. 1, 73-115.
 
  \bibitem[Co94]{Co94} A. Connes: {\it Noncommutative Geometry}, Academic Press, Inc., San Diego, CA, 1994, xiv+661 pp.
  
 \bibitem[CU18]{CU18} I. Chifan and B. Udrea: {\it Some rigidity results for II$_1$ factors arising from wreath products of property (T) groups}, preprint
 arXiv:1804.04558.
   
  \bibitem[DHI16]{DHI16} D. Drimbe, D. Hoff, and A. Ioana: {\it Prime II$_1$ factors arising from irreducible lattices in products of rank one simple Lie groups}, preprint arXiv:1611.02209, to appear in J. Reine. Angew. Math.
  
  \bibitem[DL69]{DL69} J. Dixmier and E. C. Lance: {\it Deux nouveaux facteurs de type II$_1$}, Invent. Math. {\bf 7} (1969), 226-234.
  
   \bibitem[FGL06]{FGL06} J. Fang, L. Ge and W. Li: {\it Central sequence algebras of von Neumann algebras}, Taiwanese J. Math. {\bf 10} (2006), 187-200.
  
  
  \bibitem[FHS09]{FHS09} I. Farah, B. Hart, and D. Sherman: {\it Model theory of operator algebras I: stability}, Bull. Lond. Math. Soc. {\bf 45} (2013), 825-838.
  
  
  \bibitem[FHS11]{FHS11} I. Farah, B. Hart, and D. Sherman: {\it Model theory of operator algebras III: elementary equivalence and II$_1$ factors},
Bull. Lond. Math. Soc. {\bf 46} (2014), 609-628.
  
  
\bibitem[Gl03]{Gl03} E. Glasner: {\it Ergodic theory via joinings}, Amer. Math. Society, 2003.
  
    \bibitem[GH16]{GH16} I. Goldbring and B. Hart: {\it On the theories of McDuff's II$_1$ factors},  Int. Math. Res. Not., Volume 27, Issue 18 (2017), 5609-5628.
  
  \bibitem[Ha79]{Ha79} U. Haagerup: {\it An example of a non-nuclear C$^*$-algebra, which has the metric approximation property}, Invent.
Math. {\bf 50} (1978/79), 279-293.
  
  
  
  \bibitem[HS16]{HS16} D. Hadwin and T. Shulman: {\it Tracial Stability for C$^*$-Algebras}, Integral Equations Operator Theory, 90(1):90:1, 2018.
  
  \bibitem[HS17]{HS17} D. Hadwin and T. Shulman: {\it Stability of group relations under small Hilbert-Schmidt perturbations}, J. Funct. Anal. {\bf 275} (2018), no. 4, 761-792.
  
  
  \bibitem[Ho12a]{Ho12b} C. Houdayer: {\it Structure of II$_1$ factors arising from free Bogoljubov actions of arbitrary groups}, Adv. Math. {\bf 260} (2014), 414-457.
  
\bibitem[HU15]{HU15} C. Houdayer and Y. Ueda:  {\it Rigidity of free product von Neumann algebras}, Compos. Math. {\bf 152} (2016), 2461-2492.
  
  \bibitem[Io12]{Io12} A. Ioana: {\it Cartan subalgebras of amalgamated free product II$_1$ factors}, With an appendix by Adrian Ioana and Stefaan Vaes. Ann. Sci. \'{E}c. Norm. Sup\'{e}r. (4) {\bf 48} (2015), no. 1, 71-130.
  
  \bibitem[Io18]{Io18} A. Ioana: {\it Rigidity for von Neumann algebras}, Proc. Int. Cong. of Math. 2018 Rio
de Janeiro, Vol. 2 (1635-1668).
  
  \bibitem[Is16]{Is16} Y. Isono: {\it On fundamental groups of tensor product II$_1$ factors}, preprint arXiv:1608.06426,
to appear in J. Inst. Math. Jussieu. 
  
  \bibitem[IS18]{IS18} A. Ioana and P. Spaas: {\it A class of II$_1$ factors with a unique McDuff decomposition}, preprint arXiv: 1808.02965.
  
  \bibitem[JS85]{JS85} V.F.R. Jones and K. Schmidt: {\it Asymptotically invariant sequences and approximate finiteness}, Amer. J. Math. {\bf 109} (1987), no. 1, 91-114.
  
  \bibitem [Ma17]{Ma17} A. Marrakchi: {\it Stability of products of equivalence relations}, Compos. Math. {\bf 154} (2018), no. 9, 2005-2019.
    
  \bibitem[Mc69a]{Mc69a} D. McDuff: {\it A countable infinity of II$_1$ factors}, Ann. of Math. (2) {\bf 90} (1969), 361-371.
  
  \bibitem[Mc69b]{Mc69b} D. McDuff: {\it Uncountably many II$_1$ factors}, Ann. of Math. (2) {\bf 90} (1969), 372-377.
  
  \bibitem[Mc69c]{Mc69c} D. McDuff: {\it Central sequences and the hyperfinite factor}, Proc. London Math. Soc. (3) {\bf 21} (1970), 443-461.
  
  \bibitem[Mc69d]{Mc69d} D. McDuff: {\it On residual sequences in a II$_1$ factor}, J. London Math. Soc. (2) (1971), 273-280.
  
  \bibitem[MvN43]{MvN43} F. Murray and J. von Neumann: {\it Rings of operators}, IV, Ann. of Math. {\bf 44} (1943), 716-808.
  
  \bibitem[OP07]{OP07} N. Ozawa and S. Popa: {\it On a class of II$_1$ factors with at most one Cartan subalgebra}, Ann. of Math. (2) {\bf 172} (2010), no. 1, 713-749.
  
  \bibitem[Oz03]{Oz03} N. Ozawa: \textit{Solid von Neumann algebras}, Acta Math. \textbf{192} (2004), 111-117.
  
  \bibitem[PS09]{PS09} J. Peterson and T. Sinclair: {\it On cocycle superrigidity for Gaussian actions}, Erg. Th. Dyn. Sys., {\bf 32} (2012), 249-272.
 
  \bibitem[Po86]{Po86} S. Popa: {\it Correspondences}, INCREST preprint {\bf 56} (1986), available at \url{www.math.ucla.edu/~popa/preprints.html}.

\bibitem[Po01]{Po01} S. Popa: {\it Some rigidity results for non-commutative Bernoulli shifts}, J. Funct. Anal. {\bf 230} (2006), 273-328.

  \bibitem[Po03]{Po03} S. Popa: {\it Strong rigidity of II$_1$ factors arising from malleable actions of {\it w}-rigid groups. I.}, Invent. Math. {\bf 165} (2006), no. 2, 369-408.
  
  \bibitem[Po06a]{Po06a} S. Popa: {\it On the superrigidity of malleable actions with spectral gap}, J. Amer. Math. Soc. {\bf 21} (2008), 981-1000.
  
  \bibitem[Po06b]{Po06b} S. Popa: {\it On Ozawa's property for free group factors}, Int. Math. Res. Not. 2007, no. 11, Art. ID rnm036, 10
pp.
  
  \bibitem[Po07]{Po07} S. Popa: \textit{Deformation and rigidity for group actions and von Neumann algebras}, In Proceedings of the International Congress of Mathematicians (Madrid, 2006), Vol. I, European Mathematical Society Publishing House, 2007, p. 445-477.

  
  \bibitem[Po09a]{Po09a} S. Popa: {\it On spectral gap rigidity and Connes invariant $\chi(M)$}, Proc. Amer. Math. Soc. {\bf 138} (2010), no. 10, 3531-3539.
 
 \bibitem[Po09b]{Po09b} S. Popa: {\it On the classification of inductive limits of II$_1$ factors with spectral gap}, Trans. Amer. Math. Soc. {\bf 364} (2012), 2987-3000.
 
 
  \bibitem[Sa62]{Sa62} S. Sakai: {\it The Theory of W$^*$-Algebras}, lecture notes, Yale University, 1962.
  
  
\bibitem[Sa68]{Sa68} S. Sakai: {\it Asymptotically abelian II$_1$-factors}, Publ. Res. Inst. Math. Sci. Ser. A 4 1968/1969 299-307.
  
  
  \bibitem[Si10]{Si10} T. Sinclair: {\it Strong solidity of group factors from lattices in SO(n,1) and SU(n,1)}, J. Funct. Anal., {\bf 260} (2011), 3209-3221.
  
 \bibitem[Th18]{Th18} A. Thom: {\it Finitary approximations of groups and their applications}, Proc. Int. Cong. of Math. 2018 Rio
de Janeiro, Vol. 2 (1775-1796).
  
\bibitem[Va10a]{Va10a} S. Vaes: \textit{Rigidity for von Neumann algebras and their invariants}, Proceedings of the ICM (Hyderabad, India, 2010), Vol. III, Hindustan Book Agency (2010),  1624-1650.
  
  \bibitem[Va10b]{Va10b} S. Vaes: {\it One-cohomology and the uniqueness of the group measure space of a II$_1$ factor}, Math. Ann., {\bf 355} (2013), 661-696.
  
  \bibitem[VDN92]{VDN92} D.-V. Voiculescu, K. J. Dykema, and A. Nica: {\it Free Random Variables}, CRM Monogr. Ser. vol.1, AMS, Providence, RI, 1992.
  
  \bibitem[Wr54]{Wr54} F. B. Wright: {\it A reduction for algebras of finite type}, Ann. of Math. {\bf 60} (1944), 560-570.
 
 \bibitem[ZM69]{ZM69} G. Zeller-Meier: {\it Deux autres facteurs de type II$_1$}, (French) Invent. Math. {\bf 7} (1969) 235-242.
 
 
\end{thebibliography}
\end{document}